\documentclass{article}
\usepackage{amsmath, amsfonts,amsthm, amssymb}
\usepackage{url,hyperref}
\usepackage[english]{babel}

\newtheorem{theorem}{Theorem}[section]
\newtheorem{assumption}{Assumption}[section]

\newtheorem{proposition}{Proposition}[section]
\newtheorem{remark}{Remark}[section]
\newtheorem{lemma}{Lemma}[section]

\title{Short-time asymptotics for marginal distributions of semimartingales}

\date{January 2012}

\author{Amel Bentata\  and Rama Cont}
\begin{document}

\maketitle

\begin{abstract}
We study the short-time asymptotics of conditional expectations of
smooth and non-smooth functions of a  (discontinuous) Ito
semimartingale;  we compute the leading term in the asymptotics   in
terms of the local characteristics of the semimartingale. We derive
in particular the asymptotic behavior of call options with short
maturity in a semimartingale model: whereas the behavior of
\textit{out-of-the-money} options is found to be linear in time,
the short time asymptotics of \textit{at-the-money} options is shown
to depend on the fine structure of the semimartingale.
\end{abstract}
\tableofcontents
\newpage

\section{Introduction}

In applications such as stochastic control, statistics of processes
and mathematical finance, one is often interested in computing or
approximating conditional expectations of the type
\begin{equation}\label{chp4.eq.cond.margi}
\mathbb{E}\left[f(\xi_{t})|\mathcal{F}_{t_0}\right]
\end{equation}
where $\xi$ is a stochastic process. Whereas for Markov process
various well-known tools --partial differential equations, Monte
Carlo simulation, semigroup methods-- are available for the
computation and approximation of conditional expectations, such
tools do not carry over to the more general setting of
semimartingales. Even in the Markov case, if the state space is high
dimensional exact computations may be computationally prohibitive
and there has been a lot of interest in obtaining approximations of
\eqref{chp4.eq.cond.margi} as $t\to t_0$. Knowledge of such { \it
short-time asymptotics}  is very useful not only for computation of
conditional expectations but also for the estimation and calibration
of such models. Accordingly, short-time asymptotics for
\eqref{chp4.eq.cond.margi}  (which, in the Markov case, amounts to
studying transition densities of the process $\xi$) has been
previously studied for diffusion models \cite{bbf,bbf2,feng10},
L\'evy processes
\cite{jacod07,leandre87,ruschendorf02,bnhubalek08,lopez09,lopez12,tankov11},
Markov jump-diffusion models \cite{alos07,benhamou09} and
one-dimensional martingales \cite{nutz11}, using a variety of
techniques. The proofs of these results in the case of   L\'evy
processes makes heavy use of the independence of increments; proofs
in other case rely on the Markov property, estimates for heat
kernels for second-order differential operators or Malliavin
calculus. What is striking, however, is the similarity of the
results obtained in these different settings.

We reconsider here  the  short-time asymptotics of conditional
expectations in a more general framework which contains existing
models but allows to go beyond the Markovian setting and  to
incorporate  path-dependent features.  Such a framework is provided
by  the class of  {\it It\^{o} semimartingales}, which contains all
the examples cited above but allows the use the tools of stochastic
analysis. An  {\it It\^{o} semimartingale}  on a filtered
probability space $(\Omega,\mathcal{F},(\mathcal{F}_t)_{t\geq
0},\mathbb{P})$ is a stochastic process $\xi$ with the
representation
\begin{equation}\label{chp4.classeJ}
  \xi_t=\xi_0+\int_0^t \beta_s\,ds+\int_0^t \delta_s\,dW_s
  +\int_0^t\int_{\mathbb{R}^d} \kappa(y)\,\tilde{M}(ds\:dy)+ \int_0^t\int_{\mathbb{R}^d}\left(y-\kappa(y)\right)\,{M}(ds\:dy),
\end{equation}
where $\xi_0$ is in $\mathbb{R}^d$, $W$ is a standard
$\mathbb{R}^n$-valued Wiener process, $M$ is an integer-valued
random measure  on
  $[0,\infty]\times\mathbb{R}^d$ with compensator $\mu(\omega,dt,dy)=m(\omega,t,dy) dt$ and
$\tilde{M}=M-\mu$ its compensated random measure,  $\beta$ (resp.
$\delta$) is an adapted process with values in $\mathbb{R}^d$ (resp.
$M_{d\times n}(\mathbb{R})$) and
\begin{equation*}
\kappa(y)= \frac{y}{1+\|y\|^2}
\end{equation*}
is a truncation function.

 We study the short-time asymptotics of conditional expectations of the form \eqref{chp4.eq.cond.margi}  where $\xi$ is an Ito  semimartingale of the form \eqref{chp4.classeJ}, for various classes of functions $f:\mathbb{R}^d\to\mathbb{R}$.
First, we  prove a general result for the case of $f\in
C^2_b(\mathbb{R}^d,\mathbb{R})$. Then we will treat, when $d=1$, the
case of
\begin{equation}\label{chp4.eq.call.margi}
\mathbb{E}\left[(\xi_{t}-K)^+|\mathcal{F}_{t_0}\right],
\end{equation}
which 
corresponds to the value at $t_0$ of a call option with strike $K$
and maturity $t$ in a  model described by equation
(\ref{chp4.classeJ}). We show that whereas the behavior of
\eqref{chp4.eq.call.margi} in the case $K> \xi_{t_0}$ (
\textit{out-of-the-money} options) is linear in $t-t_0$,  the
asymptotics  in the case $K=\xi_{t_0}$ (which corresponds to
\textit{at-the-money} options) depends on the fine structure of the
semimartingale $\xi$ at $t_0$.  In particular, we show that for
continuous semimartingales the short-maturity asymptotics of
at-the-money options is determined by the local time of $\xi$ at
$t_0$.
In each case we identify the leading term in the asymptotics and
express this term in terms of the local characteristics of the
semimartingale at $t_0$.

Our results  unify various asymptotic results previously derived for
particular examples of stochastic models and extend them to  the
more general  case of a discontinuous semimartingale. In particular,
we show that the independence of increments or the Markov property
do not play any role in the derivation of such results.

Short-time asymptotics for expectations of the form
\eqref{chp4.eq.cond.margi} have been studied in the context of
statistics of processes \cite{jacod07,bnhubalek08}  and option
pricing
\cite{alos07,bbf,bbf2,benhamou09,gatheral2011,lopez09,lopez12,tankov11,nutz11}.
Berestycki, Busca and Florent \cite{bbf,bbf2} and Gatheral et al
\cite{gatheral2011} derive short maturity asymptotics for call
options when $\xi_t$ is a diffusion, using analytical methods.
Durrleman \cite{durrleman10} studied the asymptotics of implied
volatility in a   stochastic volatility model. Jacod \cite{jacod07}
derived asymptotics for (\ref{chp4.eq.cond.margi}) for various
classes of functions $f$, when $\xi_t$ is a L\'{e}vy process.
Figueroa-Lopez and Forde \cite{lopez09} and Tankov \cite{tankov11}
study the asymptotics of (\ref{chp4.eq.call.margi}) when $\xi_{t}$
is the exponential of a L\'{e}vy process.  Figueroa-Lopez and
Houdr\'e \cite{lopez09} also studies short-time asymptotic
expansions for (\ref{chp4.eq.cond.margi}), by iterating the
infinitesimal generator of the  L\'{e}vy process $\xi_t$.
Figueroa-Lopez and Forde \cite{lopez09}  extend these results and
derive a second order small-time expansion for out-of-the-money call
options under an exponential L\'{e}vy  model. Alos et al
\cite{alos07}  derive short-maturity expansions for call options and
implied volatility in a Heston model using Malliavin calculus.
Benhamou et al. \cite{benhamou09} derive short-maturity expansions
for call options in a model where $\xi$ is a Markov process whose
jumps are described by a compound Poisson process. More generally,
these results apply to processes with independent  increments or
Markov processes expressed as the solution of a stochastic
differential equation with regular coefficients.

Durrleman studied the convergence of implied volatility to spot
volatility in a stochastic volatility model with finite-variation
jumps \cite{durrleman08}. More recently, Nutz  and Muhle-Karbe
\cite{nutz11} study short-maturity asymptotics for call options in
the case where $\xi_t$ is a one-dimensional It\^o semimartingale
driven by a (one-dimensional) Poisson random measure whose L\'evy
measure is absolutely continuous.
 Their approach consists  in ``freezing" the characteristic triplet of $\xi$ at  $t_0$, approximating $\xi_t$ by the corresponding L\'{e}vy process and using the results cited above \cite{jacod07,lopez09} to derive asymptotics for call option prices.

Our contribution is  to extend these results  to the more general
case when $\xi$ is a $d$-dimensional semimartingale with jumps. By
using minimal assumptions on the process $\xi$, we put previous
results into perspective: in contrast to previous derivations, our
approach is purely based on It\^{o} calculus and makes no use of the
Markov property or independence of increments. Also, our
multidimensional setting allows to treat examples which are not
accessible using previous results such as \cite{nutz11}. For
instance, when studying index options in jump-diffusion models, one
considers an index $I_t=\sum w_i S^i_t$ where $(S^1,...,S^d)$ are
It\^{o} semimartingales.  In this framework, $I$ is indeed an
It\^{o} semimartingale whose stochastic integral representation is
implied by those of $S^i$  but it is naturally  represented in terms
of a $d$-dimensional integer-valued random measure, not a
one-dimensional Poisson random measure. Our setting provides a
natural framework for treating such examples.

\section{Short time asymptotics for conditional expectations}
\subsection{Main result}
We make the following assumptions on the characteristics of the
semimartingale $\xi$:
\begin{assumption}[Right-continuity of characteristics at $t_0$]\label{chp4.A} \ \\
\begin{equation*}
\lim_{t\to t_0,\, t>t_0}
\mathbb{E}\left[\|\beta_t-\beta_{t_0}\||\mathcal{F}_{t_0}\right]=0,\quad\quad\lim_{t\to
t_0,\, t>t_0} \mathbb{E}\left[\|\delta_t-\delta_{t_0}\|^2
|\mathcal{F}_{t_0}\right]=0,
\end{equation*}
where $ \|.\|\ $ denotes the Euclidean norm on $\mathbb{R}^d$ and
for
$\varphi\in\mathcal{C}_0^b(\mathbb{R}^d\times\mathbb{R}^d,\mathbb{R})$,
\begin{equation*}
\lim_{t\to t_0,\, t>t_0} \mathbb{E}\left[\int_{\mathbb{R}^d}
\|y\|^2\,\varphi(\xi_{t},y)\,m(t,dy)|\mathcal{F}_{t_0}\right]=
\int_{\mathbb{R}^d}  \|y\|^2\,\varphi(\xi_{t_0},y)\,m(t_0,dy).
\end{equation*}
\end{assumption}
The second requirement, which may be viewed as a weak (right)
continuity of $m(t,dy)$ along the paths of $\xi$, is satisfied for
instance if $m(t,dy)$ is absolutely continuous with a density which
is right-continuous in $t$ at $t_0$.
\begin{assumption}[Integrability condition]\label{chp4.H}  $\exists T> t_0,$
\begin{eqnarray*}
&& \mathbb{E}\left[\int_{t_0}^T\|\beta_s\|\,ds\Big|\mathcal{F}_{t_0}\right]<\infty,\quad\quad\mathbb{E}\left[\int_{t_0}^T\|\delta_s\|^2\,ds\Big|\mathcal{F}_{t_0}\right]<\infty, \\[0.15cm]
&& \mathbb{E}\left[\int_{t_0}^T\int_{\mathbb{R}^d}
\|y\|^2\,m(s,dy)\,ds\Big|\mathcal{F}_{t_0}\right]<\infty.
\end{eqnarray*}
\end{assumption}
Under these assumptions,  the following result describes the
asymptotic behavior of
$\mathbb{E}\left[f(\xi_t)|\mathcal{F}_{t_0}\right]$ when $t \to
t_0$: \clearpage
\begin{theorem}\label{chp4.th.devpt.asympt}
Under Assumptions \ref{chp4.A} and \ref{chp4.H}, for all
$f\in\mathcal{C}_b^2(\mathbb{R}^d,\mathbb{R})$,
\begin{equation} \label{chp4.devpt.asympt}
 \lim_{t\downarrow t_0}\frac{1}{t-t_0}\,\left(\mathbb{E}\left[f(\xi_t)|\mathcal{F}_{t_0}\right]-f(\xi_{t_0})\right) = \mathcal{L}_{t_0}f(\xi_{t_0}).
\end{equation}
where $\mathcal{L}_{t_0 }$ is the (random) integro-differential
operator  given   by
  \begin{eqnarray}\label{chp4.L1.eq}
\forall f\in \mathcal{C}_b^2(\mathbb{R}^d,\mathbb{R}), \quad     \mathcal{L}_{t_0}f(x)&=&\beta_{t_0}.\nabla f(x)+\frac{1}{2}\mathrm{tr}\left[{}^t\delta_{t_0}\delta_{t_0}\nabla^2f\right](x)\\[0.1cm]
      &+&\int_{\mathbb{R}^d}[f(x+y)-f(x)-\frac{1}{1+\|y\|^2}\,y.\nabla f(x)]m(t_0,dy).\nonumber
  \end{eqnarray}
\end{theorem}
Before proving Theorem \ref{chp4.th.devpt.asympt}, we recall a
useful lemma:
\begin{lemma}\label{lemme.cad}
Let $f:\mathbb{R}^+\to\mathbb{R}$ be right-continuous at $0$, then
\begin{equation}
\lim_{t\to 0} \frac{1}{t} \int_0^t f(s)\,ds = f(0).
\end{equation}
\end{lemma}
\begin{proof}
Let $F$ denote the primitive of $f$, then
\begin{equation*}
\frac{1}{t} \int_0^t f(s)\,ds= \frac{1}{t} \left(F(t)-F(0)\right).
\end{equation*}
Letting $t\to 0^+$, this is nothing but the right derivative at $0$
of $F$, which is $f(0)$ by right continuity of $f$.
\end{proof}

We can now prove Theorem \ref{chp4.th.devpt.asympt}.
\begin{proof} of Theorem \ref{chp4.th.devpt.asympt}\\
We first note that, by replacing $\mathbb{P}$ by the conditional
measure $\mathbb{P}_{|\mathcal{F}_{t_0}}$  given
$\mathcal{F}_{t_0}$, we may replace the conditional expectation in
(\ref{chp4.devpt.asympt}) by an expectation with respect to the
marginal distribution of $\xi_t$ under
${\mathbb{P}}_{|\mathcal{F}_{t_0}}$. Thus, without loss of
generality, we put $t_0=0$ in the sequel and consider the case where
${\mathcal{F}_0}$ is the $\sigma$-algebra generated by all
$\mathbb{P}$-null sets. Let
$f\in\mathcal{C}^2_b(\mathbb{R}^d,\mathbb{R})$.
 It\^{o}'s formula yields
    \begin{eqnarray*}
      f(\xi_t)&=&f(\xi_0)+\int_0^t \nabla f(\xi_{s^-})d\xi_s^i+\frac{1}{2}\int_0^t
      {\rm tr}\left[\nabla^2 f(\xi_{s^-})\,{}^t\delta_s\delta_s\right]\,ds\\[0.1cm]
      &+&\sum_{s\leq t} \left[f(\xi_{s^-}+\Delta
        \xi_s)-f(\xi_{s^-})-\sum_{i=1}^d \frac{\partial f}{\partial x_i}(\xi_{s^-})\Delta  \xi_s^i \right]\\[0.1cm]
      &=&f(\xi_0)+\int_0^t \nabla f(\xi_{s^-}).\beta_{s}\,ds+
  \int_0^t \nabla f(\xi_{s^-}).\delta_{s}dW_s\\[0.1cm]
  &+&\frac{1}{2}\int_0^t {\rm tr}\left[\nabla^2 f(\xi_{s^-})\,{}^t\delta_s\delta_s\right]\ \,ds\\[0.1cm]
  &+&\int_0^t\int_{\mathbb{R}^d}\nabla f(\xi_{s^-}).\kappa(y)\,\tilde{M}(ds\,dy)\\[0.1cm]
  &+&\int_0^t\int_{\mathbb{R}^d}\left(f(\xi_{s^-}+y)-f(\xi_{s^-})-\kappa(y).\nabla
    f(\xi_{s^-})\right)\,M(ds\,dy).
    \end{eqnarray*}
 We note that
\begin{itemize}
\item since $\nabla f$ is bounded and given Assumption \ref{chp4.H}, $\int_0^t\int_{\mathbb{R}^d}\nabla f(\xi_{s^-}).\kappa(y)\,\tilde{M}(ds\,dy)$ is  a square-integrable martingale.
\item since $\nabla f$ is bounded and given Assumption \ref{chp4.H}, $\int_0^t \nabla f(\xi_{s^-}).\delta_{s} dW_s$ is a martingale.
\end{itemize}
Hence, taking  expectations, we obtain
\begin{eqnarray*}
  \mathbb{E}\left[f(\xi_t)\right]&=&\mathbb{E}\left[f(\xi_0)\right]+\mathbb{E}\left[\int_0^t \nabla f(\xi_{s^-}).\beta_{s}\,ds\right]+ \mathbb{E}\left[\frac{1}{2}\int_0^t {\rm tr}\left[\nabla^2 f(\xi_{s^-}){}^t\delta_s\delta_s\right]\,ds\right]\\[0.1cm]
  &+&\mathbb{E}\left[\int_0^t\int_{\mathbb{R}^d}\left(f(\xi_{s^-}+y)-f(\xi_{s^-})-\kappa(y).\nabla
    f(\xi_{s^-})\right)\,M(ds\,dy)\right]\\[0.1cm]
&=&\mathbb{E}\left[f(\xi_0)\right]+\mathbb{E}\left[\int_0^t \nabla
  f(\xi_{s^-}).\beta_{s}\,ds\right]+\mathbb{E}\left[\frac{1}{2}\int_0^t {\rm tr}\left[\nabla^2
  f(\xi_{s^-})\,{}^t\delta_s\delta_s\right]\,ds\right]\\[0.1cm]
&+&\mathbb{E}\left[\int_0^t\int_{\mathbb{R}^d}\left(f(\xi_{s^-}+y)-f(\xi_{s^-})-\kappa(y).\nabla
    f(\xi_{s^-})\right)\,m(s,dy)\,ds\right],
\end{eqnarray*}
that is
         \begin{equation}\label{chp4.dynkin.no.project}
 \mathbb{E}\left[f(\xi_t)\right]=\mathbb{E}\left[f(\xi_0)\right]+\mathbb{E}\left[\int_0^t \mathcal{L}_sf(\xi_{s})\,ds\right].
 \end{equation}
where $\mathcal{L}$ denote the integro-differential operator given,
for all $t\in [t_0,T]$ and for all $f\in
\mathcal{C}_b^2(\mathbb{R}^d,\mathbb{R})$, by
  \begin{equation}
   \begin{split}
      \mathcal{L}_{t}f(x)&=\beta_{t}.\nabla f(x)+ \frac{1}{2}\mathrm{tr}\left[{}^t\delta_{t}\delta_{t}\nabla^2f\right](x)\\[0.1cm]
      &+\int_{\mathbb{R}^d}[f(x+y)-f(x)-\frac{1}{1+\|y\|^2}\,y.\nabla f(x)]m(t,dy),
   \end{split}
  \end{equation}
Equation (\ref{chp4.dynkin.no.project}) yields
\begin{eqnarray*}
&&\frac{1}{t}\,\mathbb{E}\left[f(\xi_t)\right]-\frac{1}{t}f(\xi_0)-\mathcal{L}_0f(\xi_0)\\[0.1cm]
&=&\mathbb{E}\left[\frac{1}{t}\int_0^t ds\, \left(\nabla f(\xi_{s}).\beta_s-\nabla f(\xi_{0}).\beta_0\right)\right]\\[0.1cm]
      &+&\frac{1}{2}\,\mathbb{E}\left[\frac{1}{t}\int_0^t \,ds\,{\rm tr}\left[\nabla^2
          f(\xi_{s})\,{}^t\delta_{s}\delta_{s}-\nabla^2
          f(\xi_{0})\,{}^t\delta_{0}\delta_{0}\right]\right] \\[0.1cm]
      &+&\mathbb{E}\Big[\int_{\mathbb{R}^d}\frac{1}{t}\int_0^t\,\,ds\,\big[m(s,dy)\,\left(f(\xi_{s}+y)-f(\xi_{s})-\kappa(y).\nabla
        f(\xi_{s})\right)\\[0.1cm]
&& \quad\quad\quad\quad\quad\quad\quad\quad\quad -
m(0,dy),\left(f(\xi_{0}+y)-f(\xi_{0})-\kappa(y).\nabla
        f(\xi_{0})\right)\big]\Big].
\end{eqnarray*}
Define
\begin{eqnarray*}
\Delta_1(t)&=&\mathbb{E}\left[\frac{1}{t}\int_0^t ds\, \left(\nabla f(\xi_{s}).\beta_s-\nabla f(\xi_{0}).\beta_0\right)\right],\\[0.1cm]
\Delta_2(t)&=&\frac{1}{2}\,\mathbb{E}\left[\frac{1}t\int_0^t
\,ds\,{\rm tr}\left[\nabla^2
          f(\xi_{s^-})\,{}^t\delta_{s}\delta_{s}-\nabla^2
          f(\xi_{0})\,{}^t\delta_{0}\delta_{0}\right]\right],\\[0.1cm]
\Delta_3(t)&=&\mathbb{E}\Big[\int_{\mathbb{R}^d}\frac{1}{t}\int_0^t\,\,ds\,\big[m(s,dy)\,\left(f(\xi_{s}+y)-f(\xi_{s})-\kappa(y).\nabla
        f(\xi_{s^-})\right)\\[0.1cm]
&& \quad\quad\quad\quad\quad\quad \quad\quad\quad-
m(0,dy)\,\left(f(\xi_{0}+y)-f(\xi_{0})-\kappa(y).\nabla.
        f(\xi_{0})\right)\big]\Big].
\end{eqnarray*}
Thanks to Assumptions \ref{chp4.A} and \ref{chp4.H},
\begin{equation*}
\mathbb{E}\left[\int_0^t ds\,\left|\nabla f(\xi_{s}).\beta_s-\nabla
f(\xi_{0}).\beta_0\right|\right]\leq\mathbb{E}\left[\int_0^t
ds\,\|\nabla f\|\left( \|\beta_s\|+\|\beta_0\|\right)\right]<\infty.
\end{equation*}
 Fubini's theorem then applies:
\begin{equation*}
\Delta_1(t)=\frac{1}{t}\int_0^t ds\,\mathbb{E}\left[ \nabla
f(\xi_{s}).\beta_s-\nabla f(\xi_{0}).\beta_{0}\right].
\end{equation*}
Let us prove that
\begin{equation*}
\begin{split}
g_1:[0,T[&\to \mathbb{R}\\[0.1cm]
t&\to \mathbb{E}\left[ \nabla f(\xi_{t}).\beta_t-\nabla
f(\xi_{0}).\beta_0\right],
\end{split}
\end{equation*}
is right-continuous at $0$ with $g_1(0)=0$, yielding $\Delta_1(t)\to
0$ when $t\to 0^+$ if one applies Lemma \ref{lemme.cad}.

\begin{equation}\label{g1s}
\begin{split}
\left|g_1(t)\right|&=\left|\mathbb{E}\left[ \nabla f(\xi_{t}).\beta_t-\nabla f(\xi_{0}).\beta_{0}\right]\right|\\[0.1cm]
&=\left|\mathbb{E}\left[ \left(\nabla f(\xi_{t})-\nabla f(\xi_{0})\right).\beta_{0}+\nabla f(\xi_{t}).\left(\beta_t-\beta_{0}\right)\right]\right|\\[0.1cm]
&\leq \|\nabla
f\|_{\infty}\,\mathbb{E}\left[\|\beta_t-\beta_{0}\|\right] +
\|\beta_{0}\|\,\|\nabla^2 f\|_{\infty}\,
\mathbb{E}\left[\|\xi_{t}-\xi_{0}\|\right],
\end{split}
\end{equation}
where $\|\|_{\infty}$ denotes the supremum norm on
$\mathcal{C}^2_b(\mathbb{R}^d,\mathbb{R})$. Assumption \ref{chp4.A}
implies that:
\begin{equation*}
\lim_{t\to 0^+} \mathbb{E}\left[\|\beta_t-\beta_{0}\|\right]=0.
\end{equation*}
Thanks to Assumption \ref{chp4.H}, one may decompose $\xi_t$ as
follows
\begin{equation}
\begin{split}
\xi_t&=\xi_0+ A_t+M_t,\\[0.1cm]
A_t&= \int_0^t \left(\beta_s\,ds+\int_{\mathbb{R}^d}(y-\kappa(y))\,m(s,dy)\right)\,ds,\\[0.1cm]
M_t&=\int_0^t \delta_s\,dW_s +\int_0^t\int_{\mathbb{R}^d} y
\,\tilde{M}(ds\:dy),
\end{split}
\end{equation}
where $A_t$ is of finite variation and $M_t$ is a local martingale.
First, applying Fubini's theorem (using Assumption \ref{chp4.H}),
\begin{eqnarray*}
\mathbb{E}\left[ \|A_{t}\|\right]&\leq& \mathbb{E}\left[ \int_0^t \|\beta_s\|\,ds\right]+\mathbb{E}\left[ \int_0^t \int_{\mathbb{R}^d}\|y-\kappa(y)\|\, m(s,dy)\,ds\right]\\[0.1cm]
&=&\int_0^t ds\,\mathbb{E}\left[ \|\beta_s\|\right]+\int_0^t
\,ds\,\mathbb{E}\left[ \int_{\mathbb{R}^d}\|y-\kappa(y)\|\,
m(s,dy)\right].
\end{eqnarray*}
Thanks to Assumption \ref{chp4.A}, one observes that if
$s\in[0,T[\to \mathbb{E}\left[\|\beta_s-\beta_0\|\right]$ is
right-continuous at 0 so is $s\in[0,T[\to
\mathbb{E}\left[\|\beta_s\|\right]$. Furthermore, Assumption
\ref{chp4.A} yields that
$$s\in[0,T[\to \mathbb{E}\left[  \int_{\mathbb{R}^d}\|y-\kappa(y)\|\, m(s,dy) \right]$$
is right-continuous at 0 and Lemma \ref{lemme.cad} implies that
$$
\lim_{t\to 0^+} \mathbb{E}\left[ \|A_{t}\|\right]=0.
$$
Furthermore, writing $M_t=(M_t^1,\cdots,M_t^d)$,
\begin{equation*}
\mathbb{E}\left[\|M_t\|^2\right]= \sum_{1\leq i \leq
d}\,\mathbb{E}\left[|M_t^i|^2\right].
\end{equation*}
Burkholder's inequality \cite[Theorem IV.73]{protter} implies that
there exists $C>0$ such that
\begin{eqnarray*}
\sup_{s\in[0,t]} \mathbb{E}\left[|M_s^i|^2\right]&\leq& C\,\mathbb{E}\left[[M^i,M^i]_t\right]\\[0.1cm]
&=& C\,\mathbb{E}\left[\int_0^tds\,
|\delta_s^i|^2+\int_0^tds\,\int_{\mathbb{R}^{d}} |y_i|^2
m(s,dy)\right].
\end{eqnarray*}
Using Assumption \ref{chp4.H} we may apply Fubini's theorem to
obtain
\begin{eqnarray*}
\sup_{s\in[0,t]} \mathbb{E}\left[\|M_t\|^2\right]&\leq& C\,\sum_{1\leq i \leq d}\,\mathbb{E}\left[\int_0^tds\, |\delta_s^i|^2\right]+\mathbb{E}\left[\int_0^tds\,\int_{\mathbb{R}^{d}} |y_i|^2 m(s,dy)\right]\\[0.1cm]
&=& C\,\left(\mathbb{E}\left[\int_0^tds\, \|\delta_s\|^2\right]+\mathbb{E}\left[\int_0^tds\,\int_{\mathbb{R}^{d}} \|y\|^2 m(s,dy)\right]\right)\\[0.1cm]
&=& C\,\left(\int_0^tds\,
\mathbb{E}\left[\|\delta_s\|^2\right]+\int_0^tds\,\mathbb{E}\left[\int_{\mathbb{R}^{d}}
\|y\|^2 m(s,dy)\right]\right).
\end{eqnarray*}
Thanks to Assumption \ref{chp4.A}, Lemma \ref{lemme.cad} yields
$$
\lim_{t\to 0^+} \mathbb{E}\left[\|M_t\|^2\right]=0.
$$
Using the Jensen inequality, one obtains
\begin{equation*}
\mathbb{E}\left[\|M_t\|\right]=\mathbb{E}\left[\sqrt{\sum_{1\leq i
\leq d} |M_t^i|^2}\right]\leq\sqrt{\mathbb{E}\left[\sum_{1\leq i
\leq d} |M_t^i|^2\right]} =\mathbb{E}\left[\|M_t\|^2\right].
\end{equation*}
Hence,
$$
\lim_{t\to 0^+} \mathbb{E}\left[\|M_t\|\right]=0,
$$
and
$$
\lim_{t\to 0^+} \mathbb{E}\left[\|\xi_{t}-\xi_{0}\|\right]\leq
\lim_{t\to 0^+} \mathbb{E}\left[\|A_t\|\right] +\lim_{t\to 0^+}
\mathbb{E}\left[\|M_t\|\right]=0.
$$
Going back to the inequalities (\ref{g1s}), one obtains
\begin{equation*}
\lim_{t\to 0^+} g_1(t)= 0.
\end{equation*}
Similarly, $\Delta_2(t)\to 0$ and $\Delta_3(t)\to 0$ when $t\to
0^+$. This ends the proof.
\end{proof}

\begin{remark}
In applications where a process is constructed as the solution to a
stochastic differential equation driven by a Brownian motion and a
Poisson random measure, one usually starts from a representation of
the form
\begin{equation}\label{chp4.classe.K} 
 \zeta_t=\zeta_0+\int_0^t \beta_s\,ds+\int_0^t \delta_s\,dW_s+\int_0^t\int\psi_s(y)\,\tilde{N}(ds\:dy), 
\end{equation}
where $\xi_0\in\mathbb{R}^d$, $W$ is a standard
$\mathbb{R}^n$-valued Wiener process, $\beta$ and $\delta$ are
non-anticipative c\`{a}dl\`{a}g processes, $N$ is a Poisson random
measure on $[0,T]\times\mathbb{R}^d$ with intensity   $\nu(dy)\,dt$
where $\nu$ is a L\'evy measure
\begin{equation*}
\int_{\mathbb{R}^d} \left(1\wedge \|y\|^2\right)
\nu(dy)<\infty,\qquad \tilde{N}=N-\nu(dy) dt,
\end{equation*}
and  $\psi:[0,T]\times\Omega\times \mathbb{R}^d\mapsto \mathbb{R}^d$
is a predictable random function representing jump amplitude. This
representation is different from \eqref{chp4.classeJ}, but
\cite[Lemma 2]{bentatacont09} shows that one can switch from the
representation (\ref{chp4.classe.K}) to the representation
(\ref{chp4.classeJ}) in an explicit manner.

In particular, if one rewrites Assumption \ref{chp4.A} in the
framework of equation (\ref{chp4.classe.K}), one recovers the
Assumptions of \cite{nutz11} as a special case.

\end{remark}

\begin{remark}
It is sufficient for $f$ to be locally bounded on the neighborhood
of $\xi_0$.
\end{remark}
\subsection{Some consequences and examples}




If we have further information on the behavior of $f$ in the
neighborhood of  $\xi_{0}$, then the quantity $L_{0}f(\xi_{0})$ ca
be computed  more explicitly. We summarize some commonly encountered
situations in the following Proposition.
\begin{proposition}
Under Assumptions \ref{chp4.A} and \ref{chp4.H},
\begin{enumerate}
\item If $f(\xi_0)=0$ and $\nabla f(\xi_0)=0$, then
\begin{equation}
 \lim_{t\to 0^+}\frac{1}{t}\,\mathbb{E}\left[f(\xi_t)\right] =\frac{1}{2}{\rm tr}\left[{}^t\delta_{0}\delta_{0}\,\nabla^2 f(\xi_{0})\right] +\int_{\mathbb{R}^d}f(\xi_{0}+y)\,m(0,dy).
\end{equation}
\item If furthermore $\nabla^2f(\xi_{0})=0$, then
\begin{equation}
 \lim_{t\to 0^+}\frac{1}{t}\,\mathbb{E}\left[f(\xi_t)\right]= \int_{\mathbb{R}^d} f(\xi_{0}+y)\,m(0,dy).
\end{equation}
\end{enumerate}
\end{proposition}
\begin{proof}
Applying Theorem \ref{chp4.th.devpt.asympt},
$\mathcal{L}_{0}f(\xi_{0})$ writes
  \begin{equation*}
   \begin{split}
      \mathcal{L}_{0}f(\xi_{0})&=\beta_0.\nabla f(\xi_{0})+\frac{1}{2}{\rm tr}\left[\nabla^2 f(\xi_{0})\,{}^t\delta_{0}\delta_{0}\right](\xi_0)\\[0.1cm]
      &+\int_{\mathbb{R}^d}[f(\xi_{0}+y)-f(\xi_{0})-\frac{1}{1+\|y\|^2}\,y.\nabla f(\xi_{0})]m(0,dy).
   \end{split}
  \end{equation*}
The proposition follows immediately.
\end{proof}
\begin{remark} As observed by Jacod \cite[Section 5.8]{jacod07}  in the setting of L\'{e}vy processes,
 if $f(\xi_{0})=0$ and $\nabla f(\xi_{0})=0$, then $f(x)= O(\|x-\xi_{0}\|^2)$. If furthermore $\nabla^2f(\xi_{0})=0$, then $f(x)= o(\|x-\xi_{0}\|^2)$.
\end{remark}
Let us now compute in a more explicit manner the asymptotics of
\eqref{chp4.eq.cond.margi} for specific semimartingales.

\subsubsection{Functions of a Markov process}
An important situations which often arises in applications is when a
stochastic processe $\xi$ is driven by an underlying Markov process,
i.e.
\begin{equation}
\xi_t=f(Z_t)\quad \quad f\in C^2(\mathbb{R}^d,\mathbb{R}),
\end{equation}
where $Z_t$ is a Markov process, defined as the weak solution on
$[0,T]$ of a stochastic differential equation
\begin{equation}\label{chp4.classeK}
  \begin{split}
 Z_t&=Z_0+\int_0^t b(u,Z_{u-})\,du+\int_0^t \Sigma(u,Z_{u-})\,dW_u\\[0.1cm]
    &+\int_0^t\int \psi(u,Z_{u-},y)\,\tilde{N}(du\ dy),
  \end{split}
  \end{equation}
where  $(W_t)$ is an n-dimensional Brownian motion, $N$ is a Poisson
  random measure on $[0,T]\times\mathbb{R}^d$ with L\'evy measure
  $\nu(y)\,dy$, $\tilde{N}$ the associated
compensated random measure, $\Sigma:[0,T]\times \mathbb{R}^d\mapsto
M_{d\times d}(\mathbb{R})$, $b:[0,T]\times \mathbb{R}^d\mapsto
\mathbb{R}^d$ and $\psi:[0,T]\times \mathbb{R}^d\times \mathbb{R}^d$
are measurable functions such that
\begin{equation}\label{chp4.diffeo}
\begin{split}
&\psi(.,.,0)=0\quad \quad\quad \psi(t,z,.)\:\mathrm{is\:a\,} \mathcal{C}^1(\mathbb{R}^d,\mathbb{R}^d)-{\rm diffeomorphism}\\[0.1cm]
& \forall t\in[0,T],\quad \mathbb{E}\left[\int_0^t \int_{\{\|y\|\geq
1\}} \sup_{z\in\mathbb{R}^d}\left(1\wedge
\|\psi(s,z,y)\|^2\right)\,\nu(y)\,dy\,ds\right] <\infty.
\end{split}
\end{equation}
In this setting, as shown in \cite{bentatacont09}, one may verify
the regularity assumptions
 Assumption \ref{chp4.A} and Assumption \ref{chp4.H} by requiring mild and easy-to-check assumptions on the coefficients:
\begin{assumption}\label{chp4.A.Z}
 $b(.,.)$, $\Sigma(.,.)$ and $\psi(.,.,y)$ are continuous in the neighborhood of $(0,Z_0)$ 

\end{assumption}
\begin{assumption}\label{chp4.H.Z} There exist $T>0,R>0$ such that
\begin{eqnarray*}
 &\mathrm{Either} & \quad \forall t\in [0,T] \quad \inf_{\|z-Z_0\|\leq R}\,\inf_{x\in\mathbb{R}^d,\,\|x\|=1} {}^tx.\Sigma(t,z).x>0 \\[0.1cm]
 & \mathrm{or}& \Sigma\equiv 0.
\end{eqnarray*}
\end{assumption}
We then obtain the following  result:
\begin{proposition}
Let $f\in {\mathcal C}^2_b(\mathbb{R}^d,\mathbb{R})$ such that
\begin{equation}
\forall z\in\mathbb{R}^d, \quad \frac{\partial f}{\partial
z_{d}}(z)\ne 0.
\end{equation}
Define
\begin{equation*}
  \begin{cases}
    \beta_0&=\nabla f(Z_0).b(0,Z_0)+\frac{1}{2}{\rm tr}\left[\nabla^2 f(Z_{0}){}^t\Sigma(0,Z_0)\Sigma(0,Z_0)\right]\\[0.15cm]
    &+\int_{\mathbb{R}^d}\left(f(Z_0+\psi(0,Z_0,y))-f(Z_0)-\psi(0,Z_0,y).\nabla f(Z_0)\right)\,\nu(y)\,dy,\\[0.15cm]
    \delta_0&=\|\nabla f(Z_{0})\Sigma(0,Z_0)\|,\\[0.15cm]
  \end{cases}
\end{equation*}
and the measure  $m(0,.)$ via
\begin{equation}
\begin{split}
&m(0,[u,\infty[)=\int_{\mathbb{R}^d} 1_{\{f(Z_0+\psi(0,Z_0,y))-f(Z_0)\geq u\}}\,\nu(y)\,dy \quad u>0,\\[0.10cm]
&m(0,[-\infty,u])=\int_{\mathbb{R}^d}
1_{\{f(Z_{0}+\psi(0,Z_{0},y))-f(Z_{0})\leq u\}}\,\nu(y)\,dy \quad
u<0.
\end{split}
\end{equation}
Under the Assumptions \ref{chp4.A.Z} and \ref{chp4.H.Z}, $\forall g
\in \mathcal{C}_b^2(\mathbb{R}^d,\mathbb{R})$,
\begin{equation}
\lim_{t\to 0^+}\frac{\mathbb{E}\left[g(\xi_t)\right]-g(\xi_{0})}{t}=
\beta_0\,g'(\xi_{0})+\frac{\delta_{0}^2}{2}\,g''(\xi_{0})+\int_{\mathbb{R}^d}[g(\xi_{0}+u)-g(\xi_0)-ug'(\xi_{0})]\,m(0,du).
\end{equation}
\end{proposition}
\begin{proof}
Under the conditions (\ref{chp4.diffeo}) and the Assumption
\ref{chp4.H.Z}, Proposition \ref{semi.decomp.f.de.markov} shows that
$\xi_t$ admits the semimartingale decomposition
\begin{equation*}
\xi_t=\xi_0+\int_0^t \beta_s\,ds+ \int_0^t  \delta_s \,dB_s
+\int_0^t \int u \,\tilde{K}(ds\:du),
\end{equation*}
where
\begin{equation*}
  \begin{cases}
    \beta_t&=\nabla f(Z_{t^-}).b(t,Z_{t-})+\frac{1}{2}{\rm tr}\left[\nabla^2 f(Z_{t-}){}^t\Sigma(t,Z_{t-})\Sigma(t,Z_{t-})\right]\\[0.15cm]
    &+\int_{\mathbb{R}^d}\left(f(Z_{t^-}+\psi(t,Z_{t-},y))-f(Z_{t^-})-\psi(t,Z_{t-},y).\nabla f(Z_{t^-})\right)\,\nu(y)\,dy,\\[0.15cm]
    \delta_t&=\|\nabla f(Z_{t-})\Sigma(t,Z_{t-})\|,\\[0.15cm]
  \end{cases}
\end{equation*}
and $K$ is an integer-valued random measure on
$[0,T]\times\mathbb{R}$ with compensator $k(t,Z_{t-},u)\,du\,dt$
defined via
\begin{eqnarray*}
k(t,Z_{t-},u)&=& \int_{\mathbb{R}^{d-1}} \left|{\rm det}\nabla_y\Phi(t,Z_{t-},(y_1,\cdots,y_{d-1},u))\right|\\[0.1cm]
&&\quad\quad\quad\quad\quad\quad\,\nu(\Phi(t,Z_{t-},(y_1,\cdots,y_{d-1},u)))\,dy_{1}\cdots\,dy_{d-1},
\end{eqnarray*}
with
\begin{equation*}
\begin{cases}
&\Phi(t,z,y)=\phi(t,z,\kappa_z^{-1}(y)) \quad\kappa_z^{-1}(y)=(y_1,\cdots,y_{d-1},F_z(y)),\\[0.1cm]
& F_z(y):\mathbb{R}^d\to\mathbb{R} \quad
f(z+(y_1,\cdots,y_{d-1},F_z(y)))-f(z)=y_d.
\end{cases}
\end{equation*}
From Assumption \ref{chp4.A.Z} it follows that Assumptions
\ref{chp4.A} and \ref{chp4.H} hold for $\beta_t$, $\delta_t$ and
$k(t,Z_t-,.)$ on $[0,T]$. Applying Theorem
\ref{chp4.th.devpt.asympt}, the result follows immediately.
\end{proof}
\begin{remark}
Benhamou et al. \cite{benhamou09} studied   the  case where $Z_t$ is
the solution of a `Markovian' SDE whose jumps are given by a
compound Poisson Process. The above results generalizes their result
to the (general) case where the jumps are driven by an arbitrary
integer-valued random measure.
\end{remark}
\subsubsection{Time-changed L\'evy processes}
Models based on time--changed L\'evy processes provide another class
of examples of non-Markovian models which have generated recent
interest in mathematical finance. Let $L_t$ be a real-valued L\'evy
process, $(b,\sigma^2,\nu)$ be its characteristic triplet, $N$ its
jump measure. Define
\begin{equation}
\xi_t= L_{\Theta_t} \qquad \Theta_t=\int_0^t \theta_s ds,
\end{equation}
where $(\theta_t)$ is a locally bounded $\mathcal{F}_t$-adapted
positive c\`{a}dl\`{a}g process, interpreted as the rate of time
change.
\begin{proposition}
If  \begin{equation} \int_{\mathbb{R}} |y|^2\,\nu(dy)<\infty\quad
{\rm and}\quad\lim_{t\to 0,\, t>0}
\mathbb{E}\left[|\theta_t-\theta_{0}|\right]=0\label{chp4.A.levy}
\end{equation}
then
\begin{equation}
\forall f\in\mathcal{C}_b^2(\mathbb{R},\mathbb{R}),\qquad \lim_{t\to
0^+}\frac{\mathbb{E}\left[f(\xi_t)\right]-f(\xi_{0})}{t}
 = \theta_0\,\mathcal{L}_{0}f(\xi_{0})
\end{equation}
where $ \mathcal{L}_{0}$ is the infinitesimal generator of the $L$:
  \begin{equation}
  \mathcal{L}_{0}f(x)=b\,f'(x)+\frac{\sigma^2}{2}\,f''(x)
      +\int_{\mathbb{R}^d}[f(x+y)-f(x)-\frac{1}{1+|y|^2}\,yf'(x)]\nu(dy).
  \end{equation}
\end{proposition}
\begin{proof}
Considering the L\'evy-It\^{o} decomposition of $L$:
\begin{eqnarray*}
 L_t&=&\left(b-\int_{\{|y|\leq 1\}} \left(y -\kappa(y)\right)\,\nu(dy)\right)\,t+\sigma W_t\\[0.1cm]
&+&  \int_0^t\int_{\mathbb{R}} \kappa(z) \tilde{N}(ds dz) +
\int_0^t\int_{\mathbb{R}} \left(z-\kappa(z)\right) N(ds dz),
\end{eqnarray*}
then, as shown in \cite{bentatacont09}, $\xi$ has the representation
\begin{equation*}
  \begin{split}
    \xi_t &= \xi_0+ \int_0^t \sigma \sqrt{\theta_s}\,dZ_s +
\int_0^t \left(b-\int_{\{|y|\leq 1\}} \left(y -\kappa(y)\right)\,\nu(dy)\right)\theta_s\,ds   \\[0.1cm]
    & +\int_0^{t}\int_{\mathbb{R}} \kappa(z)\theta_s \,\tilde{N}(ds\,dz) + \int_0^{t}\int_{\mathbb{R}} \left(z-\kappa(z)\right)\theta_s\, N(ds\,dz).
 \end{split}
 \end{equation*}
where $Z$ is a Brownian motion. With the notation of equation
(\ref{chp4.classeJ}), one identifies
\begin{equation*}
  \beta_t=\left(b-\int_{\{|y|\leq 1\}} \left(y -\kappa(y)\right)\,\nu(dy)\right)\,\theta_t,\quad \delta_t=\sigma\,\sqrt{\theta_t},\quad m(t,dy)=\theta_t\,\nu(dy).
\end{equation*}
If  \eqref{chp4.A.levy} holds, then Assumptions \ref{chp4.A} and
\ref{chp4.H} hold for $(\beta,\delta,m)$ and Theorem
\ref{chp4.th.devpt.asympt} may be applied to obtain the result.
\end{proof}

\section{Short-maturity asymptotics for call options}

Consider a (strictly positive)  price process $S$ whose dynamics
under the pricing measure $\mathbb{P}$ is given by a stochastic
volatility model with jumps:
\begin{equation}\label{chp4.stochmodel}
S_t = S_0 +\int_0^t r(s) S_{s^-} ds + \int_0^t S_{s^-}\delta_s dW_s
+ \int_0^t\int_{-\infty}^{+\infty} S_{s^-}(e^y - 1) \tilde{M}(ds\
dy),
\end{equation}
where  $r(t)>0$ represents a (deterministic) bounded discount rate.
For convenience, we shall assume that
$r\in\mathcal{C}_0^b(\mathbb{R}^+,\mathbb{R}^+)$. $\delta_t$
represents the volatility process and $M$ is an integer-valued
random measure with compensator $\mu(\omega;
dt\,dy)=m(\omega;t,dy)\,dt$, representing jumps in the log-price,
and $\tilde{M}=M-\mu$ its compensated random measure. We make the
following assumptions on the characteristics of $S$:
\begin{assumption}[Right-continuity  at $t_0$]\label{chp4.A.exp}
\begin{equation*}
\lim_{t\to t_0,\, t>t_0} \mathbb{E}\left[|\delta_t-\delta_{t_0}|^2
|\mathcal{F}_{t_0}\right]=0.
\end{equation*}
For all
$\varphi\in\mathcal{C}_0^b(\mathbb{R}^+\times\mathbb{R},\mathbb{R})$,
\begin{equation*}
\lim_{t\to t_0,\, t>t_0} \mathbb{E}\left[\int_{\mathbb{R}}
\left(e^{2y}\wedge|y|^2\right)\,\varphi(S_{t},y)\,m(t,dy)|\mathcal{F}_{t_0}\right]=
\int_{\mathbb{R}}
\left(e^{2y}\wedge|y|^2\right)\,\varphi(S_{t_0},y)\,m(t_0,dy).
\end{equation*}
\end{assumption}
\begin{assumption}[Integrability condition]\label{chp4.I}
 \begin{equation*}
\exists
T>t_0,\quad\mathbb{E}\left[\exp{\left(\frac{1}{2}\int_{t_0}^T\delta_s^2\,ds
        + \int_{t_0}^T ds \int_{\mathbb{R}} (e^y-1)^2 m(s,dy)\right)}| \mathcal{F}_{t_0}\right]<\infty\quad.
  \end{equation*}
\end{assumption}
We recall that the value $C_{t_0}(t,K)$ at time $t_0$ of a call
option with expiry $t>t_{0}$ and strike $K>0$ is given by
\begin{equation}\label{chp4.def.call}
C_{t_0}(t,K)=e^{-\int_{t_0}^t
r(s)\,ds}\mathbb{E}[\max(S_t-K,0)|\mathcal{F}_{t_0}].
\end{equation}
The discounted asset price
 $$\hat{S}_t=e^{-\int_{t_0}^t r(u)\,du}\,S_t,$$ is the stochastic exponential of the martingale $\xi$ defined by
$$\xi_t=\int_{0}^t \delta_s\,dW_s+\int_{0}^t\int (e^y-1)\tilde{M}(ds\,dy).$$
Under Assumption \ref{chp4.I}, we have
$$
\mathbb{E}\left[\exp{\left(\frac{1}{2}\langle \xi,\xi\rangle^d_T +
\langle \xi,\xi\rangle^c_T\right)}\right]<\infty,
$$
where $\langle \xi,\xi\rangle^c$ and $\langle \xi,\xi\rangle^d$
denote the continuous and purely discontinuous parts of $[\xi,\xi]$
and \cite[Theorem 9]{protter08} implies that
$(\hat{S}_t)_{t\in[t_0,T]}$ is a $\mathbb{P}$-martingale. In
particular the expectation in \eqref{chp4.def.call} is finite.
\subsection{Out-of-the money call options}\label{chp4.section.1.dim.case.otm}

We first study the asymptotics of out-of-the money call options i.e.
the case where $K> S_{t_0}$. The main result is as follows:
\begin{theorem}[Short-maturity behavior of out-of-the money options]\label{chp4.th.lim.call.option : OTM}
Under Assumption \ref{chp4.A.exp} and Assumption \ref{chp4.I}, if
$S_{t_0} < K$ then
\begin{equation}\label{chp4.lim.option.euro.OTM}
\frac{1}{t-t_0}\,C_{t_0}(t,K)\underset{t\to t_0^+}{\longrightarrow}
\int_{0}^{\infty}  (S_{t_0}e^y-K)_+  m(t_0, dy).
\end{equation}
\end{theorem}
This limit can also be expressed using the  exponential double tail
$ \psi_{t_0}$ of the compensator, defined as
\begin{equation}\label{chp4.def.psi.tail}
 \psi_{t_0}(z)=\int_{z}^{+\infty} dx\  e^x \int_x^{\infty} m(t_0,du) \quad z>0.
\end{equation}
Then, as shown in \cite[Lemma 1]{forward},
$$  \int_{0}^{\infty}  (S_{t_0}e^y-K)_+  m(t_0, dy)= S_{t_0}\psi_{t_0}\left(\ln{\left(\frac{K}{S_{t_0}}\right)}\right). $$
\begin{proof}
The idea is to apply Theorem \ref{chp4.th.devpt.asympt}  to smooth
approximations $f_n$ of the function $x\to(x-K)^+$ and conclude
using a dominated convergence argument.

First, as argued in the proof of Theorem \ref{chp4.th.devpt.asympt},
we put $t_0=0$ in the sequel and consider the case where
${\mathcal{F}_0}$ is the $\sigma$-algebra generated by all
$\mathbb{P}$-null sets. Applying the It\^{o} formula to
$X_t\equiv\ln{(S_t)}$, we obtain
\begin{eqnarray*}
  X_t&=&\ln{(S_0)}+\int_0^t \frac{1}{S_{s^-}}\,dS_s+\frac{1}{2}\int_0^t \frac{-1}{S_{s^-}^2}\,(S_{s-}\delta_s)^2\,ds\\[0.1cm]
  &+&\sum_{s\leq t} \left[\ln{(S_{s^-}+\Delta
      S_s)}-\ln{(S_{s^-})}-\frac{1}{S_{s^-}}\Delta S_s\right]\\[0.1cm]
  &=& \ln{(S_0)}+\int_0^t
  \left(r(s)-\frac{1}{2}\delta_s^2\right)\,ds+ \int_0^t\delta_s\,dW_s \\[0.1cm]
  &+&\int_0^t\int_{-\infty}^{+\infty} (e^y - 1)\tilde{M}(ds\:dy) -\int_0^t\int_{-\infty}^{+\infty}\left(e^y-1-y\right)\,M(ds\:dy)
\end{eqnarray*}
Note that there exists $C>0$ such that
\begin{equation*}
|e^y-1-y\,\frac{1}{1+|y|^2}|\leq C\,\left(e^y-1\right)^2.
\end{equation*}
Thanks to Jensen's inequality, Assumption \ref{chp4.I} implies that
this quantity is finite, allowing us to write
\begin{eqnarray*}
 &&\int_0^t\int_{-\infty}^{+\infty} (e^y - 1)\tilde{M}(ds\:
 dy)-\int_0^t\int_{-\infty}^{+\infty}\left(e^y-1-y\right)\,M(ds\:dy)\\[0.1cm]
&=&\int_0^t\int_{-\infty}^{+\infty} (e^y - 1)\tilde{M}(ds\:
 dy)-\int_0^t\int_{-\infty}^{+\infty}\left(e^y-1-y\,\frac{1}{1+|y|^2}\right)\,M(ds\:dy)\\[0.1cm]
&+&\int_0^t\int_{-\infty}^{+\infty} \left(y-y \frac{1}{1+|y|^2}\right)\,M(ds\:dy)\\[0.1cm]
&=&\int_0^t\int_{-\infty}^{+\infty} (e^y - 1)\tilde{M}(ds\:
 dy)-\int_0^t\int_{-\infty}^{+\infty}\left(e^y-1-y\,\frac{1}{1+|y|^2}\right)\,\tilde{M}(ds\:dy)\\[0.1cm]
&-&\int_0^t\int_{-\infty}^{+\infty}\left(e^y-1-y\,\frac{1}{1+|y|^2}\right)\,m(s,y)\,ds\:dy+\int_0^t\int_{-\infty}^{+\infty} \left(y-y \frac{1}{1+|y|^2}\right)\,M(ds\:dy)\\[0.1cm]
&=&\int_0^t\int_{-\infty}^{+\infty} y\,\frac{1}{1+|y|^2}\,\tilde{M}(ds\:dy)-\int_0^t\int_{-\infty}^{+\infty}\left(e^y-1-y\,\frac{1}{1+|y|^2}\right)\,m(s,y)\,ds\:dy\\[0.1cm]
&+&\int_0^t\int_{-\infty}^{+\infty} \left(y-y
\frac{1}{1+|y|^2}\right)\,M(ds\:dy).
\end{eqnarray*}
We can thus represent $X_t$ as in  (\ref{chp4.classeJ})):
\begin{equation}\label{chp4.stoch.log}
\begin{split}
  X_t&=X_0+\int_0^t \beta_s\,dt+\int_0^t \delta_s\,dW_s\\[0.1cm]
  &+\int_0^t\int_{-\infty}^{+\infty }y\,\frac{1}{1+|y|^2}\,\tilde{M}(ds\:dy)+ \int_0^t\int_{-\infty}^{+\infty}\left(y-y\,\frac{1}{1+|y|^2}\right)\,{M}(ds\:dy),
\end{split}
\end{equation}
with
\begin{equation*}
\beta_t=r(t)-\frac{1}{2}\delta_t^2-\int_{-\infty}^{\infty}\left(e^y-1-y\,\frac{1}{1+|y|^2}\right)\,m(t,y)\,dt\:dy.
\end{equation*}
Hence, if $\delta$ and $m(.,dy)$ satisfy Assumption \ref{chp4.A.exp}
then $\beta$, $\delta$ and $m(.,dy)$ satisfy Assumption
\ref{chp4.A}. Thanks to Jensen's inequality, Assumption \ref{chp4.I}
implies that $\beta$, $\delta$ and $m$ satisfy Assumption
\ref{chp4.H}. One may apply Theorem \ref{chp4.th.devpt.asympt} to
$X_t$ for any function of the form $f\circ exp$,
$f\in\mathcal{C}^2_b(\mathbb{R},\mathbb{R})$.
Let us introduce a family
$f_n\in\mathcal{C}^2_b(\mathbb{R},\mathbb{R})$ such that
\begin{equation*}
\begin{cases}
&f_n(x)= (x-K)^+ \quad\quad\quad |x-K|> \frac{1}{n}\\[0.1cm]
&(x-K)^+\leq f_n(x)\leq \frac{1}{n}\quad |x-K|\leq \frac{1}{n}.
\end{cases}
\end{equation*}
Then for $x\ne K$, $f_n(x)\underset{n\to \infty}{\longrightarrow}
(x-K)^+$. Define, for  $f\in C^\infty_0(\mathbb{R}^+,\mathbb{R})$,
\begin{equation}\label{chp4.integro.index}
  \begin{split}
    \mathcal{L}_0f(x)&=r(0)xf'(x)+\frac{x^2\delta_0^2}{2}f''(x)\\[0.1cm]
    &+\int_{\mathbb{R}}[f(xe^y)-f(x)-x(e^y-1).f'(x)]m(0,dy).
  \end{split}
\end{equation}
First,  observe that if $N_1\geq 1/|S_0-K|$,
$$\forall n\geq N_1, f_n(S_0)=(S_0-K)^+=0,  \quad {\rm so}$$
\begin{equation*}
\frac{1}{t}\mathbb{E}\left[(S_t-K)^+\right]\leq
\frac{1}{t}\mathbb{E}\left[f_n(S_t)\right]=\frac{1}{t}\left(\mathbb{E}\left[f_n(S_t)\right]-f_n(S_0)\right).
\end{equation*}
Letting $t\to 0^+$ yields
\begin{equation}\label{chp4.limsup}
\limsup_{t\to 0^+} \frac{1}{t}\,e^{-\int_{0}^t
r(s)\,ds}\,\mathbb{E}\left[(S_t-K)^+\right]\leq \mathcal{L}_0
f_n(S_0).
\end{equation}
Furthermore,
\begin{eqnarray*}
\mathbb{E}\left[(S_t-K)^+\right]&\geq& \mathbb{E}\left[f_n(S_t)1_{\{|S_t-K|>\frac{1}{n}\}}\right]\\[0.1cm]
&=&\mathbb{E}\left[f_n(S_t)\right]-\mathbb{E}\left[f_n(S_t)1_{\{|S_t-K|\leq \frac{1}{n}\}}\right]\\[0.1cm]
&\geq&\mathbb{E}\left[f_n(S_t)\right]-f_n(S_0)-\frac{1}{n}\mathbb{E}\left[1_{\{|S_t-K|\leq
\frac{1}{n}\}}\right].
\end{eqnarray*}
But
\begin{eqnarray*}
\mathbb{E}\left[1_{\{|S_t-K|\leq \frac{1}{n}\}}\right]&\leq& \mathbb{P}\left(S_t-K\geq  -\frac{1}{n}\right )\\[0.1cm]
&\leq&\mathbb{P}\left(S_t-S_0\geq K-S_0-\frac{1}{n}\right).
\end{eqnarray*}
There exists $N_2\geq 0$ such that for all $n\geq N_2$,
\begin{eqnarray*}
\mathbb{P}\left(S_t-S_0\geq K-S_0-\frac{1}{n}\right)&\leq &\mathbb{P}\left(S_t-S_0\geq \frac{K-S_0}{2}\right)\\[0.1cm]
&\leq&\left(\frac{2}{K-S_0}\right)^2\,\mathbb{E}\left[(S_t-S_0)^2\right],
\end{eqnarray*}
by the Bienaym\'{e}-Chebyshev inequality. Hence,
\begin{equation*}
\frac{1}{t}\mathbb{E}\left[(S_t-K)^+\right]\geq\frac{1}{t}\left(\mathbb{E}\left[f_n(S_t)\right]-f_n(S_0)\right)
-\frac{1}{n}\,\left(\frac{2}{K-S_0}\right)^2\,\frac{1}{t}\mathbb{E}\left[\phi(S_t)-\phi(S_0)\right],
\end{equation*}
with $\phi(x)=(x-S_0)^2$. Applying Theorem
\ref{chp4.th.devpt.asympt} yields
\begin{equation*}
\liminf_{t\to 0^+} \frac{1}{t}\,e^{-\int_{0}^t
r(s)\,ds}\,\mathbb{E}\left[(S_t-K)^+\right]\geq\mathcal{L}_0f_n(S_0)
-\frac{1}{n}\,\left(\frac{2}{K-S_0}\right)^2\,\mathcal{L}_0\phi(S_0).
\end{equation*}
Letting $n\to +\infty$,
\begin{equation*}
\lim_{t\to 0^+} \frac{1}{t}\,e^{-\int_{0}^t
r(s)\,ds}\,\mathbb{E}\left[(S_t-K)^+\right]=\lim_{n\to\infty}\mathcal{L}_0f_n(S_0).
\end{equation*}
Since  $S_0< K$, $f_n=0$ in a neighborhood of $S_0$ for $n\geq N_1$
so $f_n(S_0)=f''_n(S_0)=f'_n(S_0)=0$ and $\mathcal{L}_0f_n(S_0)$
reduces to
$$\mathcal{L}_0f_n(S_0)= \int_{\mathbb{R}}[f_n( S_0e^y)-f_n(S_0)]m(0,dy).$$
 A dominated convergence argument then yields
$$\lim_{n\to\infty}\mathcal{L}_0f_n(S_0) = \int_{\mathbb{R}}[ ( S_0e^y-K)_+-( S_0-K)_+ ]m(0,dy).$$
Using  integration by parts, this last expression may be rewritten
\cite[Lemma 1]{forward} as
$$ S_0\psi_{0}\left(\ln{\left(\frac{K}{S_0}\right)}\right) $$
where $\psi_{0}$ is given by \eqref{chp4.def.psi.tail}. This ends
the proof.
\end{proof}


\begin{remark}
Theorem \ref{chp4.th.lim.call.option : OTM} also applies to
\textit{in-the-money} options, with a slight modification: for $K<
S_{t_0}$,
\begin{equation}
\frac{1}{t-t_0}\,\left(C_{t_0}(t,K)-(S_{t_0}-K)\right)\underset{t\to
t_0^+}{\longrightarrow} r(t_0)\,S_{t_0}+
S_{t_0}\psi_{t_0}\left(\ln{\left(\frac{K}{S_{t_0}}\right)}\right),
\end{equation}
where
\begin{equation}
\psi_{t_0}(z)=\int_{-\infty}^z dx\  e^x \int_{-\infty}^x
m(t_0,du),\qquad{\rm for}\quad z<0
\end{equation}
denotes the exponential double tail of $m(0,.)$.
\end{remark}



\subsection{At-the-money call options}\label{chp4.section.1.dim.case.atm}

When $S_{t_0}= K$, Theorem \ref{chp4.th.lim.call.option : OTM} does
not apply. Indeed, as already noted in the case of L\'evy processes
by   Tankov \cite{tankov11} and Figueroa-Lopez and Forde
\cite{lopez12}, the short maturity behavior of at-the-money options
depends on whether a continuous martingale component  is present
and, in absence of such a component,  on the degree of activity of
small jumps, measured by the Blumenthal-Getoor index of the L\'evy
measure which measures its singularity at zero \cite{jacod07}. We
will show here that similar results hold in the semimartingale case.
We distinguish three cases:
\begin{enumerate}
\item $S$ is a pure jump process of finite variation: in this case at-the-money call options behave linearly in $t-t_0$ (Proposition \ref{chp4.prop.atm.pure.jump}).
\item $S$ is a pure jump process  of infinite variation and its small jumps resemble those of an $\alpha$-stable process: in this case at-the-money call options have an asymptotic behavior of order $|t-t_0|^{1/\alpha}$ when $t-t_0\to 0^+$ (Proposition \ref{chp4.prop.atm.pure.jump.infini}).
\item $S$ has a  continuous martingale component which is non-degenerate in the neighborhood of $t_0$: in this case  at-the-money call options are of order  $\sqrt{t-t_0}$  as $t\to t_0^+$, whether or not jumps are present (Theorem \ref{chp4.th.lim.call.option : ATM}).
\end{enumerate}
These statements are made precise in the sequel. We observe that,
contrarily to the case of out-of-the money options where the
presence of jumps dominates the asymptotic behavior, for
at-the-money options the presence or absence of a continuous
martingale (Brownian) component dominates the asymptotic behavior.


For the finite variation case, we use a slightly modified version of
Assumption \ref{chp4.A.exp}:
\begin{assumption}[Weak right-continuity of jump compensator]\label{chp4.A.exp.fv}
For all
$\varphi\in\mathcal{C}_0^b(\mathbb{R}^+\times\mathbb{R},\mathbb{R})$,
\begin{equation*}
\lim_{t\to t_0,\, t>t_0} \mathbb{E}\left[\int_{\mathbb{R}}
\left(e^{2y}\wedge|y|\right)\,\varphi(S_{t},y)\,m(t,dy)|\mathcal{F}_{t_0}\right]=
\int_{\mathbb{R}}
\left(e^{2y}\wedge|y|\right)\,\varphi(S_{t_0},y)\,m(t_0,dy).
\end{equation*}
\end{assumption}
\begin{proposition}[Asymptotic for ATM call options  for pure jump processes of finite variation]\label{chp4.prop.atm.pure.jump}
Consider the process
\begin{equation}\label{sto.model.finite.variation}
S_t = S_0 + \int_0^t \,r(s)S_{s-}\,ds+
\int_0^t\int_{-\infty}^{+\infty} S_{s^-}(e^y - 1) \tilde{M}(ds\ dy).
\end{equation}
Under the Assumptions \ref{chp4.A.exp.fv} and \ref{chp4.I} and the
condition,
\begin{equation}\label{chp4.varation.finie}
\forall t\in[t_0,T],\quad\int_{\mathbb{R}} |y|\,m(t,dy)<\infty,
\end{equation}
\begin{equation}
\frac{1}{t-t_0}\,C_{t_0}(t,S_{t_0})\underset{t\to
t_0^+}{\longrightarrow} \quad S_{t_0}\,\int_{\mathbb{R}}
(e^y-1)^+\,m(t_0,dy).
\end{equation}
\end{proposition}
\begin{proof}
Replacing $\mathbb{P}$ by the conditional probability
$\mathbb{P}_{\mathcal{F}_{t_0}}$, we  may set $t_0=0$ in the sequel
and consider the case where ${\mathcal{F}_0}$ is the
$\sigma$-algebra generated by all $\mathbb{P}$-null sets. The
Tanaka-Meyer formula  applied to $({S}_t-S_0)^+$ gives
\begin{eqnarray*}
 (S_{t}-S_0)^+&=&\int_0^t ds\, 1_{ \{S_{s-}> S_0\}} S_{s-}\,\left(r(s)-\int_{\mathbb{R}} (e^y - 1)\,m(s,dy)\right) \\[0.1cm]
&+&\sum_{0< s\leq t} (S_{s}-S_0)^+-(S_{s-}-S_0)^+.
\end{eqnarray*}
Hence, applying Fubini's theorem,
\begin{eqnarray*}
 \mathbb{E}\left[(S_{t}-S_0)^+\right]&=& \mathbb{E}\left[\int_0^t ds\,1_{ \{S_{s-}> S_0\}} S_{s-}\,\left(r(s)-\int_{\mathbb{R}} (e^y - 1) m(s,dy)\right)\right]\\[0.1cm]
&+& \mathbb{E}\left[ \int_0^t\int_{\mathbb{R}} \left[(S_{s-}e^y-S_0)^+-(S_{s-}-S_0)^+\right]\,m(s,dy)\,ds\right]\\[0.1cm]
&=&\int_0^tds\, \mathbb{E}\left[1_{ \{S_{s}> S_0\}} S_{s}\,\left(r(s)-\int_{\mathbb{R}} (e^y - 1) m(s,dy)\right)\right]\\[0.1cm]
&+&  \int_0^tds\,\mathbb{E}\left[\int_{\mathbb{R}}
\left[(S_{s}e^y-S_0)^+-(S_{s}-S_0)^+\right]\,m(s,dy)\right].
\end{eqnarray*}
Since $\hat{S}$ is a martingale,
\begin{equation}
 \mathbb{E}\left[S_t\right]=e^{\int_0^t r(s)\,ds}\,S_0<\infty.
\end{equation}
Hence $t\to\mathbb{E}\left[S_t\right]$ is right-continuous at $0$:
\begin{equation}
\lim_{t\to 0^+} \mathbb{E}\left[S_t\right]=S_0.
\end{equation}
Furthermore, under the Assumptions \ref{chp4.A} and \ref{chp4.H} for
$X_t=\log(S_t)$ (see equation (\ref{chp4.stoch.log})), one may apply
Theorem \ref{chp4.th.devpt.asympt} to the function
$$f:x\in\mathbb{R}\mapsto (\exp(x)-S_0)^2 ,$$
yielding
\begin{equation*}
 \lim_{t\to 0^+}\frac{1}{t}\,\mathbb{E}\left[(S_t-S_0)^2\right]= \mathcal{L}_{0}f(X_{0}),
\end{equation*}
where $\mathcal{L}_0$ is defined via equation (\ref{chp4.L1.eq}).
Since $\mathcal{L}_0 f(X_0)<\infty$, then in particular,
\begin{equation*}
t\mapsto\mathbb{E}\left[(S_t-S_0)^2\right]
\end{equation*}
is right-continuous at $0$ with right limit $0$. Let us show that
$$
t\in[0,T[\mapsto \mathbb{E}\left[ S_{t}\,1_{ \{S_{t}>
S_0\}}\,\left(r(t)-\int_{\mathbb{R}} (e^y - 1) m(t,dy)\right)\right]
$$
is right-continuous at $0$ with right limit $0$. Then applying Lemma
\ref{lemme.cad} yields
\begin{equation*}
\lim_{t\to 0^+} \frac{1}{t}\,\mathbb{E}\left[\int_0^{t} ds\,
S_{s}\,1_{ \{S_{s}> S_0\}}\,\left(r(s)-\int_{\mathbb{R}} (e^y - 1)
m(s,dy)\right)\right]=0.
\end{equation*}
Observing that
$$
S_t \,1_{\{ S_{t}> S_{0}\}} = (S_t-S_0)^+ + S_0\,1_{\{S_t>S_0\}},
$$
we write
\begin{eqnarray*}
&&\left|\mathbb{E}\left[ S_{t}\,1_{ \{S_{t}> S_0\}}\,\left(r(t)-\int_{\mathbb{R}} (e^y - 1) m(t,dy)\right)\right]\right|\\[0.1cm]
&=&\left|\mathbb{E}\left[ \left((S_t-S_0)^+ + S_0\,1_{\{S_t>S_0\}} \right)\,\left(r(t)-\int_{\mathbb{R}} (e^y - 1) m(t,dy)\right)\right]\right|\\[0.15cm]
&\leq& \|r\|_{\infty}\,\mathbb{E}\left[\left|S_t-S_0\right|\right]+ \|r\|_{\infty}\,\mathbb{P}\left( S_t>S_0\right)\\[0.1cm]
&+&\mathbb{E}\left[\left(S_t-S_0\right)^2\right]^{1/2}\,\mathbb{E}\left[\int_{\mathbb{R}} \left(e^y - 1\right)^2\, m(t,dy)\right]^{1/2}\\[0.1cm]
&+&S_0^2\,\mathbb{P}\left[S_t>S_0\right]^{1/2}\,\mathbb{E}\left[\int_{\mathbb{R}} \left(e^y - 1\right)^2\, m(t,dy)\right]^{1/2},\\[0.1cm]
\end{eqnarray*}
using the Lipschitz continuity of $x\mapsto (x-S_0)_+$ and the
Cauchy-Schwarz inequality. Since $S$ is c\`{a}dl\`{a}g,
$$
 \lim_{t\downarrow 0} \mathbb{P}\left(S_t>S_0\right)=0,
$$
and Assumption \ref{chp4.A.exp.fv} implies that
\begin{equation*}
 \lim_{t\downarrow 0} \mathbb{E}\left[\int_{\mathbb{R}} \left(e^y - 1\right)^2\, m(t,dy)\right]= \int_{\mathbb{R}} \left(e^y - 1\right)^2\, m(0,dy)<\infty.
\end{equation*}
Letting $t\to 0^+$ in the above inequalities yields the result.

Let us now focus on the jump term and show that
\begin{equation*}
 t\in[0,T[ \mapsto \mathbb{E}\left[\int_{\mathbb{R}} \left[(S_{t}e^y-S_0)^+-(S_{t}-S_0)^+\right]\,m(t,dy)\right],
\end{equation*}
is right-continuous at $0$ with right-limit
\begin{equation*}
S_0\,\int_{\mathbb{R}} (e^y-1)^+\,m(0,dy).
\end{equation*}
One shall simply observes that
\begin{equation*}
\left|(xe^y-S_0)^+-(x-S_0)^+-(S_{0}e^y-S_0)^+\right|\leq
(x+S_0)\,|e^y-1|,
\end{equation*}
using the Lipschitz continuity of $x\mapsto (x-S_0)_+$ and apply
Assumption \ref{chp4.A.exp.fv}. This ends the proof.
\end{proof}

\begin{proposition}[Asymptotics of ATM  call options  for pure-jump martingales of infinite variation]\label{chp4.prop.atm.pure.jump.infini}
Consider a semimartingale whose continuous martingale part is zero:
\begin{equation}\label{chp4.purejump}
S_t = S_0 + \int_0^t \,r(s)S_{s-}\,ds+
\int_0^t\int_{-\infty}^{+\infty} S_{s^-}(e^y - 1) \tilde{M}(ds\ dy).
\end{equation}
Under the Assumptions \ref{chp4.A.exp} and \ref{chp4.I}, if there
exists $\alpha\in]1,2[$ and a family $m^\alpha(t,dy)$ of positive
measures such that
\begin{equation}\label{chp4.varation.infinie}
\forall t\in[t_0,T],\quad
m(\omega,t,dy)=m^\alpha(\omega,t,dy)+1_{|y|\leq
1}\frac{c(y)}{|y|^{1+\alpha}}\,dy\:\mathrm{a.s.},
\end{equation}
where $c(.)>0$ is continuous at $0$ and
\begin{equation}
\forall t\in[t_0,T]\quad\int_{\mathbb{R}}
|y|\,m^{\alpha}(t,dy)<\infty,
\end{equation}
then
\begin{equation}
\frac{1}{\left(t-t_0\right)^{1/\alpha}}\,C_{t_0}(t,S_{t_0})\underset{t\to
t_0^+}{\longrightarrow} \quad
S_{t_0}\,\frac{1}{2\pi}\int_{-\infty}^\infty \frac{e^{-c(0)\
|z|^\alpha}-1}{z^2}\,dz.
\end{equation}
\end{proposition}
\begin{proof}
Without loss of generality, we set $t_0=0$ in the sequel and
consider the case where ${\mathcal{F}_0}$ is the $\sigma$-algebra
generated by all $\mathbb{P}$-null sets. The at-the-money call price
can be expressed as
\begin{equation}
C_0(t,S_0)=\mathbb{E}\left[(S_t-S_0)^+\right]=S_0\mathbb{E}\left[\left(\frac{S_t}{S_0}-1\right)^+\right].
\end{equation}
Define, for $f\in C^2_b(]0,\infty[,\mathbb{R})$
\begin{equation}
  \begin{split}
    \mathcal{L}_0f(x)&=r(0)xf'(x)+\int_{\mathbb{R}}[f(xe^y)-f(x)-x(e^y-1).f'(x)]m(0,dy).
  \end{split}
\end{equation}
We decompose $\mathcal{L}_0$ as the sum
$\mathcal{L}_0=\mathcal{K}_0+\mathcal{J}_0$ where
\begin{equation*}
  \begin{split}
    \mathcal{K}_0f(x)&=r(0)xf'(x)+\int_{\mathbb{R}}[f(xe^y)-f(x)-x(e^y-1).f'(x)]\,m^{\alpha}(0,dy),\\[0.1cm]
    \mathcal{J}_0f(x)&=\int_{-1}^1 [f(xe^y)-f(x)-x(e^y-1).f'(x)]\,\frac{c(y)}{|y|^{1+\alpha}}\,dy.
  \end{split}
\end{equation*}
The term $\mathcal{K}_0$ may  be  be interpreted in terms of Theorem
\ref{chp4.th.devpt.asympt}: if $(Z_t)_{[0,T]}$ is a {\it finite
variation}  semimartingale of the form \eqref{chp4.purejump}
starting from $Z_0=S_0$ with jump compensator
  $m^{\alpha}(t,dy) $, then by Theorem \ref{chp4.th.devpt.asympt},
\begin{equation}
\forall f\in C^2_b(]0,\infty[,\mathbb{R}), \qquad\lim_{t\to
0^+}\frac{1}{t}\,e^{-\int_{0}^t r(s)\,ds}\, \mathbb{E}\left[ f(Z_t)
\right] = \mathcal{K}_0f(S_0).
\end{equation}
The idea is now to interpret
$\mathcal{L}_0=\mathcal{K}_0+\mathcal{J}_0$ in terms of a {\it
multiplicative decomposition}  $S_t=Y_t Z_t$ where $Y={\mathcal
E}(L)$ is the  stochastic exponential  of a pure-jump L\'evy process
with L\'evy measure $c(y)/|y|^{1+\alpha}\,dy$, which we can take
independent from $Z$. Indeed, let $Y={\mathcal E}(L)$ where $L$ is a
pure-jump L\'evy martingale with L\'evy measure $1_{|y|\leq 1}\
c(y)/|y|^{1+\alpha}\  dy$,  independent from $Z$, with infinitesimal
generator $\mathcal{J}_0$. Then $Y$ is a martingale  and $[Y,Z]=0$.
Then $S=YZ$ and $Y$ is an exponential L\'evy martingale, independent
from $Z$, with $E[Y_t]=1$.

A result of Tankov \cite[Proposition 5, Proof 2]{tankov11} for
exponential L\'evy processes  then implies that
\begin{equation}
\frac{1}{t^{1/\alpha}}\,\,\mathbb{E}\left[(Y_t-1)^+\right]\mathop{\to}^{t\to
0^+}\frac{1}{2\pi}\int_{-\infty}^\infty \frac{e^{-c(0)
|z|^\alpha}-1}{z^2}\,dz.\label{eq.tankov}
\end{equation}
We will show that the term \eqref{eq.tankov} is the dominant term
which gives the asymptotic behavior of $C_0(T,S_{0})$.

Indeed, by the Lipschitz continuity of $x\mapsto (x-S_0)_+$,
$$ | (S_t-S_0)_+-S_0(Y_t-1)_+|\leq Y_t|Z_t-S_0|,$$
so, taking expectations and using that $Y$ is independent from $Z$,
we get
$$ \underbrace{\mathbb{E}[ e^{-\int_{0}^t r(s)\,ds}| (S_t-S_0)_+}_{C_0(t,S_0)}-S_0(Y_t-1)_+|]\leq \underbrace{\mathbb{E}(Y_t)}_{=1} \mathbb{E}[ e^{-\int_{0}^t r(s)\,ds}|Z_t-S_0|].$$
To estimate the right hand side of this inequality note that
$|Z_t-S_0|=(Z_t-S_0)_+ +(S_0-Z_t)_+$. Since $Z$ has finite
variation, from Proposition \ref{chp4.prop.atm.pure.jump}
$$E[ e^{-\int_{0}^t r(s)\,ds}(Z_t-S_0)_+]\mathop{\sim}^{t\to 0^+} t S_{0}\int_{0}^{\infty} dx\ e^x  m([x,+\infty[).$$
Using  the martingale property of $ e^{-\int_{0}^t r(s)\,ds}Z_t$)
yields
$$E[ e^{-\int_{0}^t r(s)\,ds}(S_0-Z_t)_+]\mathop{\sim}^{t\to 0^+} t S_{0}\int_{0}^{\infty} dx\ e^x  m([x,+\infty[).$$
Hence, dividing by $t^{1/\alpha}$ and taking $t\to 0^+$ we obtain
\begin{equation*}
\frac{1}{t^{1/\alpha}}\,e^{-\int_{0}^t
r(s)\,ds}\,\mathbb{E}\left[|Z_t-S_0|^+\right] \mathop{\to}^{t\to
0^+} 0.
\end{equation*}
Thus, dividing by $t^{1/\alpha}$ the above inequality and using
\eqref{eq.tankov} yields
\begin{eqnarray*}
\frac{1}{t^{1/\alpha}}\,e^{-\int_{0}^t r(s)\,ds}\,
\mathbb{E}\left[(S_t-S_0)_+ \right]\mathop{\to}^{t\to 0^+}
S_0\,\frac{1}{2\pi}\int_{-\infty}^\infty \frac{e^{-c(0)
|z|^\alpha}-1}{z^2}\,dz.
\end{eqnarray*}
\end{proof}
We now focus on a third case, when $S$ is a continuous
semimartingale, i.e. an Ito process. From known results in the
diffusion case \cite{busca02}, we expect in this case a
short-maturity behavior ins $O(\sqrt{t})$.
We propose here  a proof of this behavior in a semimartingale setting using the notion of semimartingale local time.
\begin{proposition}[Asymptotic for at-the-money options  for continuous semimartingales]
 \label{chp4.prop.atm.diffus}
Consider the  process
\begin{equation}\label{diffusion}
S_t = S_0 +\int_0^t r(s) S_{s} ds + \int_0^t S_{s}\delta_s dW_s .
\end{equation}
Under the Assumptions \ref{chp4.A.exp} and   \ref{chp4.I}  and the
following non-degeneracy condition in the neighborhood of $t_0$,
$$\exists \epsilon >0, \qquad \mathbb{P}\left(  \forall t\in[t_0,T],\qquad \delta_t \geq \epsilon\right)=1,$$
we have
\begin{equation}
\frac{1}{\sqrt{t-t_0}}\,C_{t_0}(t,S_{t_0})\underset{t\to t_0^+}
{\longrightarrow} \frac{S_{t_0}}{\sqrt{2\pi}}\,\delta_{t_0}.
\end{equation}
\end{proposition}
\begin{proof}
Set $t_0=0$  and consider, without loss of generality, the case
where ${\mathcal{F}_0}$ is the $\sigma$-algebra generated by all
$\mathbb{P}$-null sets. Applying the Tanaka-Meyer formula  to
$(S_t-S_0)^+$, we have
\begin{equation*}
  (S_{t}-S_0)^+= \int_0^{t} 1_{ \{S_{s}> S_0\}} dS_s+ \frac{1}{2}L^{S_0}_{t}(S).
\end{equation*}
where $L^{S_0}_t(S)$ corresponds to the semimartingale local time of
$S_t$ at level $S_0$ under $\mathbb{P}$. As noted in Section
\ref{chp4.section.1.dim.case.otm}, Assumption \ref{chp4.I} implies
that the discounted  price $\hat{S}_t=e^{-\int_0^t r(s)\,ds}S_t$ is
a $\mathbb{P}$-martingale. So
$$dS_t= e^{\int_0^t r(s)\,ds}\,\left(r(t)S_{t} dt+ d\hat{S}_t\right), \quad {\rm and}$$
$$
  \int_0^{t} 1_{ \{S_{s}> S_0\}} dS_s = \int_0^{t} e^{\int_0^s r(u)\,du}\,1_{ \{S_{s}> S_0\}} d\hat{S}_s+
  \int_0^{t} e^{\int_0^s r(u)\,du}\,r(s) S_{s} 1_{ \{S_{s}> S_0\}} ds,
$$
where the first term is a martingale. Taking expectations, we get:
 \begin{equation*}
 C(t,S_0)=\mathbb{E}\left[ e^{-\int_0^{t}r(s)\,ds}\int_0^{t} e^{\int_0^s r(u)\,du}\,\,r(s) S_s \,1_{\{ S_{s}> S_0\}} ds + \frac{1}{2} e^{-\int_0^{t}r(s)\,ds}\,L^{S_0}_{t}(S)\right] .
  \end{equation*}
Since $\hat{S}$ is a martingale,
\begin{equation}
\forall t\in[0,T]\quad \mathbb{E}\left[S_t\right]=e^{\int_0^t
r(s)\,ds}\,S_0<\infty.
\end{equation}
Hence $t\to\mathbb{E}\left[S_t\right]$ is right-continuous at $0$:
\begin{equation}
\lim_{t\to 0^+} \mathbb{E}\left[S_t\right]=S_0.
\end{equation}
Furthermore, under the Assumptions \ref{chp4.A} and \ref{chp4.H} for
$X_t=\log(S_t)$ (see equation (\ref{chp4.stoch.log})), one may apply
Theorem \ref{chp4.th.devpt.asympt} to the function
$$f:x\in\mathbb{R}\mapsto (\exp(x)-S_0)^2 ,$$
yielding
\begin{equation*}
 \lim_{t\to 0^+}\frac{1}{t}\,\mathbb{E}\left[(S_t-S_0)^2\right]= \mathcal{L}_{0}f(X_{0}),
\end{equation*}
where $\mathcal{L}_0$ is defined via equation (\ref{chp4.L1.eq})
with $m\equiv 0$. Since $\mathcal{L}_0 f(X_0)<\infty$, then in
particular,
\begin{equation*}
t\mapsto\mathbb{E}\left[(S_t-S_0)^2\right]
\end{equation*}
is right-continuous at $0$ with right limit $0$. Let us show that
$$
t\in[0,T[\mapsto \mathbb{E}\left[ S_{t}\,1_{ \{S_{t}> S_0\}}\right]
$$
is right-continuous at $0$ with right limit $0$. Then applying Lemma
\ref{lemme.cad} yields
\begin{equation*}
\lim_{t\to 0^+} \frac{1}{t}\,\mathbb{E}\left[\int_0^{t} ds\,
S_{s}\,1_{ \{S_{s}> S_0\}}\right]=0
\end{equation*}
Observing that
$$
S_t \,1_{\{ S_{t}> S_{0}\}} = (S_t-S_0)^+ + S_0\,1_{\{S_t>S_0\}},
$$
we write
\begin{eqnarray*}
\left|\mathbb{E}\left[ S_{t}\,1_{ \{S_{t}> S_0\}}\right]\right|&=&\left|\mathbb{E}\left[ (S_t-S_0)^+ + S_0\,1_{\{S_t>S_0\}}\right]\right|\\[0.15cm]
&\leq& \mathbb{E}\left[\left|S_t-S_0\right|\right]+ \mathbb{P}\left(
S_t>S_0\right)
\end{eqnarray*}
using the Lipschitz continuity of $x\mapsto (x-S_0)_+$. Since $S$ is
c\`{a}dl\`{a}g,
$$
 \lim_{t\downarrow 0} \mathbb{P}\left(S_t>S_0\right)=0.
$$
Letting $t\to 0$ in the above inequalities yields the result. Since
$$\mathbb{E}\left[\int_0^{t} e^{\int_0^s r(u)\,du}\,r(s) S_s \,1_{\{ S_{s}> S_0\}} ds\right]=o(t),$$
a fortiori,
$$\mathbb{E}\left[\int_0^{t} e^{\int_0^s r(u)\,du}\,r(s) S_s \,1_{\{ S_{s}> S_0\}} ds\right]=o\left(\sqrt{t}\right).
$$
Hence (if the limit exists)
 \begin{equation}\label{call.vs.local.time}
 \lim_{t\to 0}\frac{1}{\sqrt{t}}\,C(t,S_0)=\lim_{t\to 0} \frac{1}{\sqrt{t}}\,e^{-\int_0^{t}r(s)\,ds}\,\mathbb{E}\left[\frac{1}{2}\,L^{S_0}_{t}(S)\right] =\lim_{t\to 0} \frac{1}{\sqrt{t}}\,\mathbb{E}\left[\frac{1}{2}\,L^{S_0}_{t}(S)\right] .
  \end{equation}
By the Dubins-Schwarz theorem  \cite[Theorem 1.5]{revuzyor}, there
exists a Brownian motion $B$ such that $$ \forall
t<[U]_{\infty},\quad U_t=\int_0^t\delta_s\,dW_s =
B_{[U]_t}=B_{\int_0^t \delta^2_s ds}.$$
\begin{eqnarray*}
{\rm So}\quad \forall t<[U]_{\infty} \quad S_t &=&S_0\exp{\left(\int_0^t \left(r(s)-\frac{1}{2}\delta_s^2\right)\,ds+B_{[ U]_t}\right)} \\[0.1cm]
&=&S_0\exp{\left(\int_0^t
\left(r(s)-\frac{1}{2}\delta_s^2\right)\,ds+B_{\int_0^t\delta_s^2\,ds}\right)}
.
\end{eqnarray*}
The occupation time formula 
then yields, for
$\phi\in\mathcal{C}_0^\infty(\mathbb{R},\mathbb{R})$,
\begin{eqnarray*}
\int_0^\infty \phi(K) L_t^K\left(S_0\exp{\left(B_{[U]}\right)}\right) \,dK &=&\int_0^t \phi\left(S_0\exp{\left(B_{[U]_u}\right)}\right)S_0^2\exp{\left(B_{[U]_u}\right)}^2\delta_u^2\,du\\[0.1cm]
&=&\int_{-\infty}^\infty \phi(S_0 \exp(y))S_0^2\exp(y)^2
L_t^y\left(B_{[U]}\right)\,dy,
\end{eqnarray*}
where $L_t^K\left(S_0\exp{\left(B_{[U]}\right)}\right)$ (resp.
$L_t^y\left(B_{[U]}\right)$) denotes the semimartingale local time
of the process $S_0\exp{\left(B_{[U]}\right)}$ at  $K$ and (resp.
$B_{[U]}$ at  $y$). A change of variable leads to
\begin{eqnarray*}
&&\int_{-\infty}^\infty \phi(S_0\exp(y)) S_0\exp(y) L_t^{S_0e^y}\left(S_0\exp{\left(B_{[U]}\right)}\right) \,dy \\[0.1cm]
&=& \int_{-\infty}^\infty \phi(S_0 \exp(y))S_0^2\exp(y)^2
L_t^y\left(B_{[U]}\right).
\end{eqnarray*}
Hence
$$L_t^{S_0}\left(S_0\exp{\left(B_{[U]}\right)}\right)=S_0L_t^0\left(B_{[U]}\right).$$
We also have
$$L_t^0\left(B_{[U]}\right)=L_{\int_0^t\delta_s^2\,ds}^0\left(B\right),$$
where $L_{\int_0^t\delta_s^2\,ds}^0\left(B\right)$ denotes the
semimartingale local time of $B$ at  time $\int_0^t\delta_s^2\,ds$
and level $0$.
Using the scaling property of Brownian motion,
\begin{eqnarray*}
\mathbb{E}\left[L_t^{S_0}\left(S_0\,\exp{\left(B_{[U]}\right)}\right)\right]&=& S_0\,\mathbb{E}\left[L_{\int_0^t\delta_s^2\,ds}^0\left(B\right)\right]\\[0.1cm]
&=& S_0\,\mathbb{E}\left[\sqrt{\int_0^t\delta_s^2\,ds}\,
L_{1}^0\left(B\right)\right].
\end{eqnarray*}
Hence
\begin{eqnarray*}
\lim_{t\to 0^+}
\frac{1}{\sqrt{t}}\,\mathbb{E}\left[L_t^{S_0}\left(S_0\,\exp{\left(B_{[U]}\right)}\right)\right]&=&
\lim_{t\to 0^+} \frac{1}{\sqrt{t}}\,S_0\,\mathbb{E}\left[\sqrt{\int_0^t\delta_s^2\,ds}\, L_{1}^0\left(B\right)\right]\\[0.15cm]
&=&\lim_{t\to 0^+}
\,S_0\,\mathbb{E}\left[\sqrt{\frac{1}{t}\int_0^t\delta_s^2\,ds}\,
L_{1}^0\left(B\right)\right].
\end{eqnarray*}
Let us show that
\begin{equation}
\lim_{t\to 0^+}
\,S_0\,\mathbb{E}\left[\sqrt{\frac{1}{t}\int_0^t\delta_s^2\,ds}\,
L_{1}^0\left(B\right)\right]
=S_0\,\delta_0\,\mathbb{E}\left[L_{1}^0\left(B\right)\right].
\label{cv0.eq}
\end{equation}
Using the Cauchy-Schwarz inequality,
\begin{equation*}
\left|\mathbb{E}\left[\left(\sqrt{\frac{1}{t}\int_0^t\delta_s^2\,ds}-\delta_0\right)\,
L_{1}^0\left(B\right)\right]\right|
\leq
\mathbb{E}\left[L_{1}^0\left(B\right)^2\right]^{1/2}\,\mathbb{E}\left[\left(\sqrt{\frac{1}{t}\int_0^t\delta_s^2\,ds}-\delta_0\right)^2\right]^{1/2}.
\end{equation*}
The Lipschitz property of $x\to\left(\sqrt{x}-\delta_0\right)^2$ on
$[\epsilon,+\infty[$ yields
\begin{eqnarray*}
\mathbb{E}\left[\left(\sqrt{\frac{1}{t}\int_0^t\delta_s^2\,ds}-\delta_0\right)^2\right]
&\leq& c(\epsilon) \mathbb{E}\left[\left|\frac{1}{t}\,\int_0^t \left(\delta^2_s - \delta_0^2\right)\,ds\right|\right]\\[0.15cm]
&\leq&\frac{c(\epsilon)}{t}\,\int_0^t\,ds\,
\mathbb{E}\left[\left|\delta^2_s - \delta_0^2\right|\right].
\end{eqnarray*}
where $c(\epsilon)$ is the Lipschitz constant of
$x\to\left(\sqrt{x}-\delta_0\right)^2$ on $[\epsilon,+\infty[$.
Assumption \ref{chp4.A.exp} and Lemma \ref{lemme.cad} then imply
\eqref{cv0.eq}.
By L\'{e}vy's theorem for the local time of Brownian motion,
$L_1^0(B)$ has the same law as $|B_1|$, leading to
\begin{equation*}
\mathbb{E}\left[L_1^0(B)\right]=\sqrt{\frac{2}{\pi}}.
\end{equation*}
Clearly, since
$L_t^K(S)=L_t^{K}\left(S_0\exp{\left(B_{[U]}\right)}\right)$,
\begin{equation}\label{asympt.local.time.disc}
\lim_{t\to 0} \frac{1}{\sqrt{t}}
\,\mathbb{E}\left[\frac{1}{2}L^{S_0}_{t}(S)\right] =
\frac{S_0}{\sqrt{2\pi}}\,\delta_0.
\end{equation}
This ends the proof.
\end{proof}
We can now treat the case of a general It\^{o} semimartingale with
both a continuous martingale component and a jump component.
\begin{theorem}[Short-maturity asymptotics for at-the-money call options]\label{chp4.th.lim.call.option : ATM}
Consider the price process $S$ whose dynamics is given by
\begin{equation*}
S_t = S_0 +\int_0^t r(s) S_{s^-} ds + \int_0^t S_{s^-}\delta_s dW_s
+ \int_0^t\int_{-\infty}^{+\infty} S_{s^-}(e^y - 1) \tilde{M}(ds\
dy).
\end{equation*}
Under the Assumptions \ref{chp4.A.exp} and \ref{chp4.I} and the
folllowing non-degeneracy condition in the neighborhood of $t_0$
$$ \exists \epsilon >0, \qquad \mathbb{P}\left(\forall t\in[t_0,T], \quad\delta_t\geq \epsilon\right)=1,$$ we have
\begin{equation}
\frac{1}{\sqrt{t-t_0}}\,C_{t_0}(t,S_{t_0})\underset{t\to
t_0^+}{\longrightarrow} \frac{S_{t_0}}{\sqrt{2\pi}}\,\delta_{t_0}.
\end{equation}
\end{theorem}
\begin{proof}
Applying  the Tanaka-Meyer formula  to $({S}_t-S_0)^+$ , we have
\begin{equation}
\begin{split}
  (S_{t}-S_0)^+&=\int_0^t 1_{ \{S_{s-}> S_0\}} dS_s + \frac{1}{2} L^{S_0}_{t} \\[0.1cm]
  &+ \sum_{0< s\leq t} (S_{s}-S_0)^+-(S_{s-}-S_0)^+-1_{ \{S_{s-}> S_0\}}\Delta S_{s}.
\end{split}
\end{equation}
As noted above, Assumption \ref{chp4.I} implies that the discounted
price $\hat{S}_t=e^{-\int_0^t r(s)\,ds}S_t$ is a martingale under
$\mathbb{P}$. So $(S_t)$ can be expressed as $dS_t=  e^{\int_0^t
r(s)\,ds}\,\left(r(t)S_{t-} dt+ d\hat{S}_t\right)$ and
$$
  \int_0^{t} 1_{ \{S_{s-}> S_0\}} dS_s = \int_0^{t} e^{\int_0^s r(u)\,du}\,1_{ \{S_{s-}> S_0\}} d\hat{S}_s+ \int_0^{t} e^{\int_0^s r(u)\,du}\,r(s) S_{s-} 1_{ \{S_{s-}> S_0\}} ds,
$$
where the first term is a martingale. Taking expectations, we get
 \begin{eqnarray}
    e^{\int_0^{t}r(s)\,ds}C(t,S_0)&=&\mathbb{E}\left[\int_0^{t}  e^{\int_0^s r(u)\,du}\,r(s) S_s \,1_{\{ S_{s-}> S_0\}} ds + \frac{1}{2}L^{S_0}_{t}\right]  \nonumber\\[0.1cm]
    & + &  \mathbb{E}\left[ \sum_{0< s\leq t} (S_{s}-S_0)^+-(S_{s-}-S_0)^+-1_{\{ S_{s-}> S_0\}}\Delta S_{s}\right].\nonumber
  \end{eqnarray}
Since $\hat{S}$ is a martingale,
\begin{equation}
\forall t\in[0,T]\quad \mathbb{E}\left[S_t\right]=e^{\int_0^t
r(s)\,ds}\,S_0<\infty.
\end{equation}
Hence $t\to\mathbb{E}\left[S_t\right]$ is right-continuous at $0$:
\begin{equation}
\lim_{t\to 0^+} \mathbb{E}\left[S_t\right]=S_0.
\end{equation}
Furthermore, under the Assumptions \ref{chp4.A} and \ref{chp4.H} for
$X_t=\log(S_t)$ (see equation (\ref{chp4.stoch.log})), one may apply
Theorem \ref{chp4.th.devpt.asympt} to the function
$$f:x\in\mathbb{R}\mapsto (\exp(x)-S_0)^2 ,$$
yielding
\begin{equation*}
 \lim_{t\to 0^+}\frac{1}{t}\,\mathbb{E}\left[(S_t-S_0)^2\right]= \mathcal{L}_{0}f(X_{0}),
\end{equation*}
where $\mathcal{L}_0$ is defined via equation (\ref{chp4.L1.eq}).
Since $\mathcal{L}_0 f(X_0)<\infty$, then in particular,
\begin{equation*}
t\mapsto\mathbb{E}\left[(S_t-S_0)^2\right]
\end{equation*}
is right-continuous at $0$ with right limit $0$. Let us show that
$$
t\in[0,T[\mapsto \mathbb{E}\left[ S_{t}\,1_{ \{S_{t}> S_0\}}\right]
$$
is right-continuous at $0$ with right limit $0$. Then applying Lemma
\ref{lemme.cad} yields
\begin{equation*}
\lim_{t\to 0^+} \frac{1}{t}\,\mathbb{E}\left[\int_0^{t} ds\,
S_{s}\,1_{ \{S_{s}> S_0\}}\right]=0
\end{equation*}
Observing that
$$
S_t \,1_{\{ S_{t}> S_{0}\}} = (S_t-S_0)^+ + S_0\,1_{\{S_t>S_0\}},
$$
we write
\begin{eqnarray*}
\left|\mathbb{E}\left[ S_{t}\,1_{ \{S_{t}> S_0\}}\right]\right|
&\leq& \mathbb{E}\left[\left|S_t-S_0\right|\right]+
S_0\mathbb{P}\left( S_t>S_0\right)
\end{eqnarray*}
Letting $t\to 0^+$ in the above inequalities yields
$$\mathbb{E}\left[\int_0^{t} r(s) S_s \,1_{\{ S_{s-}> S_0\}} ds\right]=o(t)=o(\sqrt{t}).$$
Let us now focus on the jump part,
\begin{eqnarray}
  &&\mathbb{E}\left[\sum_{0< s\leq t} (S_{s}-S_0)^+-(S_{s-}-S_0)^+-1_{ \{S_{s-}> S_0\}} \Delta S_{s}\right] \nonumber\\[0.1cm]
  &=& \mathbb{E}\left[\int_0^{t}ds\int m(s,dx)\,(S_{s-}e^x-S_0)^+-(S_{s-}-S_0)^+-1_{\{ S_{s-}> S_0\}}S_{s-}(e^x-1)\right]\nonumber\\[0.1cm]
\end{eqnarray}
Observing that
\begin{equation*}
\left|(ze^x-S_0)^+-(z-S_0)^+-1_{\{ z> S_0\}}z(e^x-1)\right|\leq
C\,(S_0e^x-z)^2,
\end{equation*}
then, together with Assumption \ref{chp4.A.exp} and Lemma
\ref{lemme.cad} implies,
$$\mathbb{E}\left[\sum_{0< s\leq t} (S_{s}-S_0)^+-(S_{s-}-S_0)^+-1_{ \{S_{s-}> S_0\}} \Delta S_{s}\right]=O(t)=o(\sqrt{t}).$$
Since $\delta_0\geq \epsilon$, equation
(\ref{asympt.local.time.disc}) yields the result.
\end{proof}

\begin{remark} As noted by Berestycki et al \cite{bbf,bbf2} in the diffusion case, the regularity of $f$ at $S_{t_0}$ plays a crucial role in the asymptotics of $\mathbb{E}\left[f(S_t)\right]$. Theorem \ref{chp4.th.devpt.asympt} shows that $\mathbb{E}\left[f(S_t)\right]\sim ct $ for smooth functions $f$, even if $f(S_{t_0})=0$, while for  call option prices we have $\sim\sqrt{t}$ asymptotics at-the-money where the function $x\to (x-S_0)_+$ is not smooth.
\end{remark}

\begin{remark}
In the particular case of a L\'{e}vy process, Proposition
\ref{chp4.prop.atm.pure.jump}, Proposition
\ref{chp4.prop.atm.pure.jump.infini} and Theorem
\ref{chp4.th.lim.call.option : ATM} imply \cite[Proposition 5, Proof
2]{tankov11}.
\end{remark}

\def\polhk#1{\setbox0=\hbox{#1}{\ooalign{\hidewidth
  \lower1.5ex\hbox{`}\hidewidth\crcr\unhbox0}}}

\end{document}